\documentclass[10pt]{amsart}

\usepackage{geometry,graphicx,amssymb,amsmath,amsbsy,eucal,amsfonts,mathrsfs,amscd,bm,tcolorbox}
\usepackage{soul,xcolor,cancel}

\usepackage{tikz}
\usetikzlibrary{patterns}
\usepackage{subcaption}
\numberwithin{equation}{section}

\allowdisplaybreaks[3]

\newtheorem{theorem}{Theorem}[section]
\newtheorem{lemma}[theorem]{Lemma}
\newtheorem{assumption}[theorem]{Assumption}
\newtheorem{corollary}[theorem]{Corollary}

\theoremstyle{definition}

\theoremstyle{remark}
\newtheorem{remark}[theorem]{Remark}
\newtheorem{example}[theorem]{Example}

\newcommand{\eps}{\varepsilon}

\newcommand{\p}{{\partial}}

\newcommand{\bl}{\bigl\langle}
\newcommand{\br}{\bigr\rangle}

\newcommand{\bld}[1]{\boldsymbol{#1}}

\newcommand{\bn}{\bld{n}}

\newcommand{\Dt}{\Delta t}
\newcommand{\uu}{\mathsf{u}}
\newcommand{\ww}{\mathsf{w}}

\newcommand{\lla}{\mathsf{l}}

\newcommand{\Lla}{\Lambda}

\newcommand{\pdt}{\partial_{\Delta t}}
\newcommand{\pt}{\partial_{t}}
\newcommand{\dt}{\Delta t}

\newcommand{\J}{\mathsf{J}}

\def\cG{\mathcal{G}}
\def\cH{\mathcal{H}}

\def\tu{u}
\def\tw{w}
\def\tla{\lambda}

\title[Parabolic-Parabolic Interface Problems]{Estimates of Discrete Time Derivatives for the Parabolic-Parabolic Robin-Robin Coupling Method}

\author{Erik Burman$^1$}
\address{$^1$Department of Mathematics, University College London, London, UK–WC1E 6BT, United Kingdom}
\email{e.burman@ucl.ac.uk}
\author{Rebecca Durst$^2$}
\address{$^2$Argonne National Laboratory, 9700S. CassAvenue, Lemont, IL 60439, USA}
\email{rebecca\_durst@alumni.brown.edu}
\author{Miguel A. Fern\'andez$^3$}
\address{$^3$Sorbonne Universit\'e \& CNRS, UMR 7598 LJLL, 75005 Paris, France -- Inria, 75012 Paris, France}
\email{miguel.fernandez@inria.fr}
\author{Johnny Guzm\'an$^4$}
\address{$^4$Division of Applied Mathematics, Brown University, 182 George Street, Providence, RI 02912, USA}
\email{johnny\_guzman@brown.edu}

\author{Sijing Liu$^{5,*}$}
\address{$^5$Department of Mathematical Sciences, Worcester Polytechnic Institute, 100 Institute Road, Worcester, MA 01609, USA}
\address{$^*$ Corresponding author: Sijing Liu}
\email{sliu13@wpi.edu}

\usepackage[foot]{amsaddr}


\begin{document}

\setstcolor{red}
\begin{abstract}
  We consider a loosely coupled, non-iterative Robin-Robin coupling method proposed and analyzed in {\em [J. Numer. Math., 31(1):59--77, 2023]} for a parabolic-parabolic interface problem and prove estimates for the discrete time derivatives of the scalar field in different norms. When the interface is flat and perpendicular to two of the edges of the domain we prove error estimates in the $H^2$-norm. Such estimates are key ingredients to analyze a defect correction method for the parabolic-parabolic interface problem. Numerical results are shown to support our findings.
\end{abstract}

\maketitle

\section{Introduction}
Time splitting methods are popular for fluid-structure interaction (FSI) problems (see, e.g., 
\cite{Badia:103018,BURMAN2009766,BHS14,burman2014explicit,bukac-et-al-14,bukac-21,gigante-vergara-21b,gigante-vergara-21,serino-et-al-19,MINEV2022100082,BUCELLI2023112326})).   One of the first stable splitting methods was proposed by Burman and Fern{\' a}ndez \cite{BURMAN2009766}. Later the same authors \cite{burman2014explicit} developed a related method which was coined the genuine Robin-Robin splitting method. Both these methods however suffered from a coupling between the space and time discretization parameters that reduced the accuracy. This constraint was lifted by eliminating the mesh dependence of the Robin parameter in \cite{burman2022stability}. The resulting method has been analyzed in \cite{burman2022stability, burman2022fully, burman2023loosely}.  A very similar method was developed and analyzed by Buka\v{c} and Seboldt \cite{bukacSeboldt}. In \cite{burman2022fully} it was proved for the FSI problems that the method converges as $\mathcal O(\sqrt{\dt})$ 
where $\dt$ is the time step. Numerical evidence suggested that those estimates were not sharp, and nearly first-order accuracy was proved (mod possibly a logarithmic factor) in \cite{burman2023loosely} for the analogue method applied to the parabolic-parabolic and hyperbolic-parabolic problems. The analysis was extended to the FSI problem in \cite{durst-22,burman2023robin}. For parabolic-parabolic couplings, there is a rich literature on splitting schemes motivated by models of ocean-atmosphere interaction. In these models, friction forces on the interface render the physical coupling dissipative through a Robin-type coupling condition, as discussed in \cite{LT13}. This aspect has been successfully exploited in the design of splitting methods \cite{connors2009partitioned,CHL12,CH12,ZHS18,ZLCJ20,ZHS20,SBPK23,LHH24}. In our case, the coupling conditions consist of continuity of both the primal variables and the fluxes across the interface. This coupling is conservative, and hence the approach suggested in the above references fails. Instead, the splitting method uses a Robin condition for the computational coupling, which turns out to lead to an unconditionally stable algorithm. An approach for conservative fluid-fluid coupling problems was proposed in \cite{FGS14}, using  Nitsche or Robin type couplings similar to those introduced in \cite{burman2014explicit}.  In a domain decomposition framework, a splitting method based on subcycling, i.e., iterative solution, was proposed and analyzed in \cite{Bene15,Bene17}. In a similar spirit, but focusing on a multi-timestep approach, a Robin-Robin coupling for time-dependent advection--diffusion was introduced in \cite{CL19}, with numerical investigation of the stability.

It is well known that discrete time differences for time stepping methods (e.g., backward Euler method) superconverge. In particular, when the backward Euler method is applied to a parabolic problem, the first-time difference of the errors converge with order $\mathcal O((\dt)^2)$. In this paper, we prove similar results for the splitting method  \cite{burman2023loosely} applied to an interface problem. In particular, we will prove second-order convergence for the scalar fields living on the two sub-domains in the $L^2$-norm. Moreover, in special configurations (i.e., the interface is horizontal and perpendicular to two sides of the domain), we will be able to prove second-order convergence in the $H^2$-norm. Numerically, the $H^2$ second-order convergence rates seem to hold on more general configurations. This appears to be the first work where error estimates in stronger norms for splitting methods applied to interface problems have been considered.

Error estimates of derivatives of the solution are, of course, of interest in their own right in many applications. However, our main motivation for this work is the application of these estimates in the analysis of a prediction-correction method \cite{ppcorrection} that we are concurrently developing. The aim is to improve the first-order convergence of \cite{burman2023loosely} to second-order convergence in time through a defect-correction procedure \cite{bohmer1984defect}.  The method in \cite{ppcorrection} uses a prediction step, which is exactly the splitting method we analyze here. The second-order accuracy of the correction step depends on the second-order accuracy of time differences of the prediction step, which is precisely the subject of this paper. 

Our overall goal is to propose and analyze a second-order loosely coupled scheme for FSI problems. This paper represents the first step towards that goal. Based on our past experiences (see \cite{burman2023loosely,burman2022fully,burman-durst-guzman-19}), the analysis of FSI is very similar to those of hyperbolic-parabolic and parabolic-parabolic interface problems. Therefore, we begin by considering the simplest problem, namely, the parabolic-parabolic interface problem, to ensure that the proofs are more transparent while maintain some of the difficulties of FSI. In this paper, all the estimates, including the $H^2$ estimate, are crucial components in analyzing a second-order convergent correction method, which is included in \cite{ppcorrection}. We believe it is beneficial to fully understand the parabolic-parabolic interface problem first before progressing to the more complex hyperbolic-parabolic interface problem and FSI problems. Note that while the problem we are considering can be formulated using a unified approach with discontinuous coefficients, our focus is on developing a loosely coupled scheme, as we mentioned, which is why we prefer the partitioned setting.

As for a single parabolic problem, the idea to prove higher convergence for time differences is to use that the time differences satisfy a similar discrete equation with new right-hand sides that have time differences themselves. Then, one uses the error analysis for the original method to proceed. In our case, we will use the analysis provided in \cite{burman2023loosely} to do this. The main difference here is that we consider a problem with Neumann boundary conditions  on two of the sides of a square instead of pure Dirichlet boundary conditions, which were considered in \cite{burman2023loosely}. This, in fact, simplifies the analysis slightly, and additionally, one can remove the logarithmic factor that appears in \cite{burman2023loosely}. It should be mentioned that we only consider the time discrete case. The fully discrete case is  more involved.

The rest of the paper is organized as follows. In Section \ref{sec:problems}, we introduce a parabolic-parabolic interface problem and the corresponding Robin-Robin coupling method. In Section \ref{sec:stability}, we present stability results for a Robin-Robin method. Section \ref{sec:errorest} is devoted to the error estimates. Finally, we provide some numerical results in Section \ref{sec:numerics} and end with some concluding remarks in Section \ref{sec:conclude}.

\section{The Parabolic-Parabolic interface problem and Robin-Robin coupling method}\label{sec:problems}
Let $\Omega=(0,1)^2$ and suppose that $\Omega=\Omega_f \cup \Omega_s \cup \Sigma$. The interface $\Sigma$ is assumed to be a line segment that intersects $\Omega$ on the two side edges; see Figure \ref{figure:neumann}.  We let $\Gamma_{N\!e}$ denote the two side edges of $\Omega$ and we let $\Gamma_D$ be the bottom and top edges of $\Omega$. We let $\Gamma_{N\!e}^i= \Gamma_{N\!e}\cap \partial \Omega_i$ for $i=s,f$. 

\begin{figure}[ht]
\centering
\begin{tikzpicture}
\draw[thick] plot coordinates {(0,0) (4,0) (4,4) (0,4) (0,0)};
\draw[thick] plot coordinates {(0,1) (4,3) };

 \node at (2,1) {$\Omega_f$};
 \node at (2,3) {$\Omega_s$};
 \node at (4.3,3) {$\Sigma$};
 \node at (2,-0.4) {$\Gamma_D^f$};
 \node at (2,4.4) {$\Gamma_D^s$};
 \node at (-0.4,0.5) {$ \Gamma_{N\!e}^f$};
 \node at (4.4,1.6) {$ \Gamma_{N\!e}^f$};
 \node at (-0.4,2.5) {$\Gamma_{N\!e}^s$};
 \node at (4.4,3.6) {$\Gamma_{N\!e}^s$};

\end{tikzpicture}
\caption{The domains $\Omega_f$ and $\Omega_s$ with interface $\Sigma$ and Neumann boundaries.} \label{figure:neumann}
\end{figure}
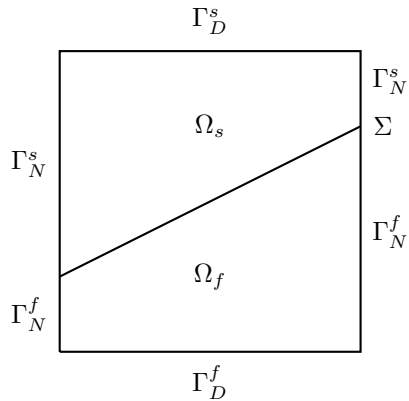

\subsection{The Parabolic-parabolic problem}
We consider the interface problem
\begin{subequations}\label{eq:ppinterface}
\begin{alignat}{2}
\pt \uu-\nu_f \Delta \uu=&0, \quad && \text{ in } [0, T] \times \Omega_f, \nonumber\\
\uu(0, x)=& \uu_0(x), \quad  && \text{ on } \Omega_f, \\
\uu=&0,  \quad && \text{ on }  [0, T] \times \Gamma^f_D,\nonumber\\
\p_{\bn} \uu =&0,  \quad && \text{ on }  [0, T] \times  \Gamma_{N\!e}^f,\nonumber\\
\nonumber\\
\pt \ww-\nu_s \Delta \ww=&0, \quad && \text{ in } [0, T] \times \Omega_s, \nonumber\\
\ww(0, x)=& \ww_0(x), \quad  && \text{ on } \Omega_s, \\
\ww=&0,  \quad && \text{ on }  [0, T] \times \Gamma^s_D,\nonumber\\
\p_{\bn} \ww =&0,  \quad && \text{ on }  [0, T] \times \Gamma_{N\!e}^s,\nonumber\\
\nonumber\\
\ww - \uu =& 0, \quad && \text{ in } [0,T] \times \Sigma, \\
\nu_s \nabla \ww \cdot \bn_s + \nu_f \nabla \uu \cdot \bn_f =& 0, \quad && \text{ in } [0,T] \times \Sigma,
\end{alignat}  
\end{subequations}
where $\bn_f$ and $\bn_s$ are the outward facing normal vectors for $\Omega_f$ and $\Omega_s$, respectively. We assume that the initial data is smooth and that $\uu$ and $\ww$ are smooth on $\Omega_f$ and $\Omega_s$, respectively.

\subsection{Variational form}
Let $(\cdot, \cdot)_i$ be the $L^2$-inner product on $\Omega_i$ for $i=f, s$. Moreover, let $\bl \cdot, \cdot \br$ be the $L^2$-inner product on $\Sigma$.  Let $N>0$ be an integer, and define $\Dt := \frac{T}{N}$, and let $\uu^{n} := \uu(t_n, \cdot)$, where $t_n := n\Dt$ for $n \in \{0, 1, 2, \dots, N\}$. We consider the spaces
\begin{subequations}\label{eq:disspaces}
  \begin{alignat}{1}
 V_f:=&\{ v \in H^1(\Omega_f): v=0 \text{ on }  \Gamma^f_D \},  \\
 V_s:= &\{ v \in H^1(\Omega_s): v=0 \text{ on  } \Gamma^s_D \},  \\
V_g:= & L^2(\Sigma).
\end{alignat}
\end{subequations}

By setting $\lla^{n+1} := \nu_f \p_{\bn_f}\uu^{n+1}$ and  assuming that $\lla^{n+1} \in L^2(\Sigma)$ for all $n$, the solution to \eqref{eq:ppinterface}  
also satisfies the following variational formulation, for $n=0,\ldots,N-1$:
\begin{subequations}\label{var}
\begin{alignat}{2}
(\partial_t \ww^{n+1},z )_s+ \nu_s (\nabla \ww^{n+1}, \nabla z)_s+ \bl \lla^{n+1}, z\br=&0 , \quad && z \in V_s,\label{var1} \\
(\partial_t \uu^{n+1}, v)_f+\nu_f (\nabla \uu^{n+1}, \nabla v)_f- \bl \lla^{n+1}, v\br=&0, \quad &&v \in V_f, \label{var2}\\
\bl \ww^{n+1}-\uu^{n+1}, \mu \br=&0, \quad && \mu \in V_g. \label{var4}
\end{alignat}
\end{subequations}

\subsection{Robin-Robin coupling: time discrete method} \label{discrete}

We define the discrete time derivatives:
\begin{alignat*}{1}
\pdt v^{n+1}=\frac{v^{n+1}-v^n}{\dt},
\end{alignat*}
and 
\begin{alignat*}{1}
\pdt^2 v^{n+1}=\frac{v^{n+1}-2 v^n+ v^{n-1}}{(\dt)^2}.
\end{alignat*}

The Robin-Robin method solves sequentially:
\begin{subequations}\label{eq1}
\begin{alignat}{2}
\pdt w^{n+1}- \nu_s \Delta w^{n+1}=&0, \quad && \text{ in }  \Omega_s, \\
w^{n+1}=&0 , \quad &&  \text{ on } \Gamma^s_D, \\
\p_{\bn} w^{n+1}=&0, \quad &&  \text{ on }   \Gamma_{N\!e}^s,\\
\alpha  w^{n+1}+ \nu_s \p_{\bn_s} w^{n+1} = &  \alpha  u^{n} - \nu_f \p_{\bn_f} u^n   \quad  && \text{ on }     \Sigma.
\end{alignat}
\end{subequations}

\begin{subequations}\label{eq2}
\begin{alignat}{2}
\pdt u^{n+1}- \nu_f\Delta u^{n+1}=&0, \quad && \text{ in } \Omega_f,   \\
u^{n+1}=&0 , \quad &&  \text{ on }  \Gamma^f_D, \\
\p_{\bn} u^{n+1}=&0, \quad &&  \text{ on }   \Gamma_{N\!e}^f,\\
\alpha u^{n+1}+ \nu_f \p_{\bn_f} u^{n+1} = & \alpha w^{n+1} + \nu_f \p_{\bn_f} u^n   \quad && \text{ on }     \Sigma.
\end{alignat}
\end{subequations}
We let $\lambda^{n+1}=\nu_f \p_{\bn_f}\tu^{n+1}$. Then, the time semi-discrete solution solves the following: Find $w^{n+1}  \in V_s$, $u^{n+1} \in V_f$, and $\lambda^{n+1} \in V_g $ such that, for $n\geq 0$,
\begin{subequations}\label{s}
\begin{alignat}{2}
(\pdt w^{n+1}, z)_s+\nu_s(\nabla w^{n+1}, \nabla z)_s + \alpha \bl w^{n+1} - u^n, z \br + \bl \lambda^n, z\br=&0 , \quad && z \in V_s, \label{s1}\\
(\pdt u^{n+1}, v)_f+\nu_f(\nabla u^{n+1}, \nabla v)_f- \bl \lambda^{n+1}, v\br=&0, \quad &&v \in V_f,  \label{s2}\\
\bl \alpha(u^{n+1}-w^{n+1})+(\lambda^{n+1}-\lambda^n), \mu \br=&0, \quad && \mu \in V_g,  \label{s3}
\end{alignat}
\end{subequations}
with $u^0=\uu_0(x)$ and $w^0=\ww_0(x)$.

One can rewrite \eqref{s3} as
\begin{equation}\label{con}
 \alpha(u^{n+1}-w^{n+1})=\lambda^n-\lambda^{n+1}, \quad \text{ on } \Sigma.
\end{equation}

\section{Stability Result}\label{sec:stability}
In this section we will prove stability results for the Robin-Robin method with a more general right-hand side: Find $w^{n+1}  \in V_s$, $u^{n+1} \in V_f$, such that, for $n\geq 0$,
\begin{subequations}\label{seq1}
\begin{alignat}{2}
\pdt \tw^{n+1}- \nu_s \Delta \tw^{n+1}=&b_1^{n+1}, \quad && \text{ in }  \Omega_s, \\
\tw^{n+1}=&0 , \quad &&  \text{ on } \Gamma^s_D, \\
\p_{\bn} \tw^{n+1}=&0, \quad &&  \text{ on }   \Gamma_{N\!e}^s,\label{eq:nbw}\\
\alpha  \tw^{n+1}+ \nu_s \p_{\bn_s} \tw^{n+1} = &  \alpha  \tu^{n} - \nu_f \p_{\bn_f} \tu^n+ \eps_1^{n+1}   \quad  && \text{ on }     \Sigma.\label{eq:rrw}
\end{alignat}
\end{subequations}

\begin{subequations}\label{seq2}
\begin{alignat}{2}
\pdt \tu^{n+1}- \nu_f\Delta \tu^{n+1}=&b_2^{n+1}, \quad && \text{ in } \Omega_f, \label{seq2.1} \\
\tu^{n+1}=&0 , \quad &&  \text{ on }  \Gamma^f_D, \\
\p_{\bn} \tu^{n+1}=&0, \quad &&  \text{ on }    \Gamma_{N\!e}^f,\label{eq:nbu}\\
\alpha \tu^{n+1}+ \nu_f \p_{\bn_f} \tu^{n+1} = & \alpha \tw^{n+1} + \nu_f \p_{\bn_f} \tu^n + \eps_2^{n+1}  \quad && \text{ on }     \Sigma,\label{eq:rru}
\end{alignat}
\end{subequations}
with initial conditions
\begin{equation}\label{eq:uwinitial}
   w^0=0\quad \text{and}\quad u^0=0,
\end{equation}
 where $b^{n+1}_1, b^{n+1}_2, \varepsilon^{n+1}_1$ and $\varepsilon^{n+1}_2$ are general right-hand side terms that are sufficiently smooth in space and time.
When analyzing the error of the Robin-Robin method the terms $\{b_i^n\}$, $\{\eps_i^n\}$ will be the residual terms.

\begin{remark}[Regularity]
We assume that $u^{n+1}\in H^2(\Omega_f)$ and $w^{n+1}\in H^2(\Omega_s)$ for the remainder of this paper.  The primary reason we assume this $H^2$ regularity and subsequently prove an $H^2$ estimate in Corollary \ref{corH2} is that, in the analysis of the correction methods, we apply the trace inequality to the Lagrange multiplier on the interface $\Sigma$. Then the $H^2$ regularity assumption is useful in analyzing the resulting quantities. It would be interesting to remove this assumption, but we currently do not have the means to do so.
\end{remark}

We let $\tla^{n+1}=\nu_f \p_{\bn_f}\tu^{n+1}$ and we see that the above solution satisfies, for $n\ge 0$,
\begin{subequations}\label{stweak}
\begin{alignat}{2}
(\pdt \tw^{n+1}, z)_s+\nu_s(\nabla \tw^{n+1}, \nabla z)_s + \alpha \bl \tw^{n+1} - \tu^n, z \br + \bl \tla^n, z\br=& L_1(z) , \quad && z \in V_s, \label{stweak1}\\
(\pdt \tu^{n+1}, v)_f+\nu_f(\nabla \tu^{n+1}, \nabla v)_f- \bl \tla^{n+1}, v\br=&L_2(v), \quad &&v \in V_f,  \label{stweak2}\\
\bl \alpha(\tu^{n+1}-\tw^{n+1})+(\tla^{n+1}-\tla^n), \mu \br=&L_3(\mu), \quad && \mu \in V_g,  \label{stweak3}
\end{alignat}
\end{subequations}
where 
\begin{alignat*}{1}
L_1(z) :=&  ( b_1^{n+1}, z)_s + \bl \eps_1^{n+1}, z \br  ,\\
L_2(v):=&(b_2^{n+1},v)_f,\\
L_3(\mu):=&  \bl \eps_2^{n+1}, \mu \br.
\end{alignat*}
For convenience we rewrite \eqref{stweak3} as:
\begin{equation}\label{errorStrong}
\alpha(\tu^{n+1} - \tw^{n+1} ) = \tla^n - \tla^{n+1}+ \eps_2^{n+1}, \quad \text{ on } \Sigma.
\end{equation}

We will need to define the following quantities for the stability estimates.
\begin{alignat*}{1}
Z^{n+1}(\psi,\phi,\theta):= & \frac{1}{2} \|\phi^{n+1}\|_{L^2(\Omega_f)}^2+\frac{1}{2}\|\psi^{n+1}\|_{L^2(\Omega_s)}^2 +\frac{\dt \alpha}{2} \|\phi^{n+1}\|_{L^2(\Sigma)}^2+ \frac{\dt}{2\alpha} \|\theta^{n+1}\|_{L^2(\Sigma)}^2, \\ 
S^{n+1}(\psi, \phi,\theta):= & \dt(\nu_f  \|\nabla \phi^{n+1}\|_{L^2(\Omega_f)}^2+  \nu_s \|\nabla \psi^{n+1}\|_{L^2(\Omega_s)}^2) +  \frac{1}{2} (\|\psi^{n+1}-\psi^n\|_{L^2(\Omega_s)}^2 + \|\phi^{n+1}-\phi^n\|_{L^2(\Omega_f)}^2)\\
& + \frac{\alpha \dt}{2}\|\phi^{n+1}-\phi^n+  \frac{1}{\alpha} ( \theta^{n+1}-\theta^n)\|_{L^2(\Sigma)}^2.
\end{alignat*}

We first state a preliminary result.
\begin{lemma}\label{errorLemma1}
Let $\tw, \tu$ solve \eqref{seq1} and \eqref{seq2} then the following identity holds. 
\begin{alignat}{1}\label{eq:errorlemma1}
Z^{n+1}(\tw, \tu, \tla)+S^{n+1}(\tw, \tu, \tla)=Z^{n}(\tw, \tu, \tla) + \dt F^{n+1}(\tw, \tu)  +\frac{\dt}{\alpha} \bl \eps_2^{n+1}, \tla^{n+1} \br,
\end{alignat}
where
\begin{alignat*}{1}
F^{n+1}(\tw, \tu):= & (b_1^{n+1}, \tw^{n+1})_s+(b_2^{n+1}, \tu^{n+1})_f + \bl \eps_1^{n+1}+\eps_2^{n+1}, \tw^{n+1} \br + \bl \tu^{n+1}-\tu^n, \eps_2^{n+1} \br. 
\end{alignat*}
\end{lemma}

\begin{proof}
To begin, we set $z= \dt \,\tw^{n+1}$ in \eqref{stweak1} and $v = \dt \,\tu^{n+1}$ in \eqref{stweak2}  
to get
\begin{alignat}{1}
& \frac{1}{2}\|\tw^{n+1}\|_{L^2(\Omega_s)}^2+ \frac{1}{2} \|\tu^{n+1}\|_{L^2(\Omega_f)}^2 +  \frac{1}{2}\|\tw^{n+1}-\tw^n\|_{L^2(\Omega_s)}^2+ \frac{1}{2} \|\tu^{n+1}-\tu^n\|_{L^2(\Omega_f)}^2 \nonumber  \\
 & + \nu_s\dt \|\nabla \tw^{n+1}\|_{L^2(\Omega_s)}^2+  \nu_f \dt \|\nabla \tu^{n+1}\|_{L^2(\Omega_f)}^2 \label{erroraux213} \\
  & =  \frac{1}{2}\|\tw^{n}\|_{L^2(\Omega_s)}^2+ \frac{1}{2} \|\tu^{n}\|_{L^2(\Omega_f)}^2+\dt \J^{n+1},\nonumber 
\end{alignat}
where
\begin{equation} \label{eq:Jn1}
\begin{aligned}
\J^{n+1}:=& -  \alpha \bl \tw^{n+1} - \tu^n, \tw^{n+1} \br- \bl  \tla^n, \tw^{n+1}\br+ \bl \tla^{n+1}, \tu^{n+1}\br +L_1(\tw^{n+1})+L_2(\tu^{n+1}).
 \end{aligned}
\end{equation}

Manipulating the first three terms in \eqref{eq:Jn1} and using \eqref{errorStrong}, we   obtain
\begin{multline}\label{eq:pre}
-  \alpha \bl \tw^{n+1} - \tu^n, \tw^{n+1} \br- \bl  \tla^n, \tw^{n+1}\br+ \bl \tla^{n+1}, \tu^{n+1}\br \\
= \mathbb{J}^{n+1} + \frac{1}{\alpha} \bl \eps_2^{n+1}, \tla^{n+1} \br + \bl \eps_2^{n+1}, \tw^{n+1} \br- \bl \tu^n-\tu^{n+1}, \eps_2^{n+1} \br,
\end{multline}
with 
\begin{equation*}
 \mathbb{J}^{n+1} := \alpha \bl \tu^n-\tu^{n+1}, \tu^{n+1} \br+ \frac{1}{\alpha} \bl  \tla^n-\tla^{n+1}, \tla^{n+1} \br- \bl \tu^n-\tu^{n+1}, \tla^n-\tla^{n+1} \br.
\end{equation*}
One can easily show that 
\begin{alignat*}{1}
 \mathbb{J}^{n+1}= & \frac{\alpha}{2}(\|\tu^n\|_{L^2(\Sigma)}^2 - \|\tu^{n+1}\|_{L^2(\Sigma)}^2)  + \frac{1}{2\alpha}(\|\tla^n\|_{L^2(\Sigma)}^2 - \|\tla^{n+1}\|_{L^2(\Sigma)}^2) \\
&- \frac{\alpha}{2}\|(\tu^n-\tu^{n+1})+ \frac{1}{\alpha}(\tla^n-\tla^{n+1}) \|_{L^2(\Sigma)}^2.
\end{alignat*}
By combining this identity with  \eqref{eq:Jn1} and \eqref{eq:pre}, we arrive at
\begin{alignat*}{1}
\J^{n+1}=&  \frac{\alpha}{2}(\|\tu^n\|_{L^2(\Sigma)}^2 - \|\tu^{n+1}\|_{L^2(\Sigma)}^2)  + \frac{1}{2\alpha}(\|\tla^n\|_{L^2(\Sigma)}^2 - \|\tla^{n+1}\|_{L^2(\Sigma)}^2) \\
&- \frac{\alpha}{2}\|(\tu^n-\tu^{n+1})+ \frac{1}{\alpha}(\tla^n-\tla^{n+1}) \|_{L^2(\Sigma)}^2\\
& +L_1(\tw^{n+1}) +L_2(\tu^{n+1}) + \frac{1}{\alpha} \bl \eps_2^{n+1}, \tla^{n+1} \br + \bl \eps_2^{n+1}, \tw^{n+1} \br+\bl  \eps_2^{n+1}, \tu^{n+1}-\tu^{n} \br.
\end{alignat*}

If we plug in these results to \eqref{erroraux213} we arrive at the identity. 
\end{proof}

We can now state an identity for the last term in \eqref{eq:errorlemma1}.
\begin{lemma}\label{lemmasummation} 
Let $\tw, \tu$ solve \eqref{seq1} and \eqref{seq2} and assuming that $\eps_2^{n+1} \in V_f$ then the following identity holds
\begin{alignat*}{1}
 \frac{\dt}{\alpha} \sum_{n=0}^{N-1} \bl \eps_2^{n+1}, \tla^{n+1} \br= & -  \frac{\dt}{\alpha} \sum_{n=1}^{N-1}  (\tu^n, \pdt \eps_2^{n+1})_f+\frac{\dt}{\alpha}\sum_{n=0}^{N-1} \Big(\nu_f (\nabla \tu^{n+1},  \nabla \eps_2^{n+1})_f- (b_2^{n+1}, \eps_2^{n+1})_f \Big) \\
 &+ \frac{1}{\alpha} (\tu^N, \eps_2^N)_f.
 \end{alignat*}
\end{lemma}

\begin{proof}
    We first take $v = \Delta t \eps_2^{n+1} $ in \eqref{stweak2} to obtain 
\begin{alignat*}{1}
 \dt \bl \eps_2^{n+1}, \tla^{n+1} \br= (\tu^{n+1}-\tu^n, \eps_2^{n+1})_f +\dt \nu_f(\nabla \tu^{n+1},  \nabla \eps_2^{n+1})_f- \dt (b_2^{n+1}, \eps_2^{n+1})_f
\end{alignat*}
If we take the sum over $n=0,\ldots,N-1$ and use summation by parts, we get
\begin{alignat*}{1}
 \dt \sum_{n=0}^{N-1} \bl \eps_2^{n+1}, \tla^{n+1} \br = &  - \dt \sum_{n=1}^{N-1} (\tu^n, \pdt \eps_2^{n+1})_f+\dt\sum_{n=0}^{N-1} \Big( \nu_f (\nabla \tu^{n+1},  \nabla \eps_2^{n+1})_f- (b_2^{n+1}, \eps_2^{n+1})_f \Big) \\
 &+ (\tu^N, \eps_2^N)_f-(\tu^0, \eps_2^{1})_f.
\end{alignat*}
We conclude the proof by using \eqref{eq:uwinitial}.
\end{proof}

To state the stability estimate we need the next definition.
\begin{alignat*}{1}
\Xi_m^N(m_1, m_2, s_1, s_2):=  &\dt \sum_{n=m}^{N-1} \left[\frac{1}{\nu_s}\| m_1^{n+1}\|_{L^2(\Omega_s)}^2+\left(\frac{1}{\nu_f}+\frac{1}{\alpha}\right)\| m_2^{n+1}\|_{L^2(\Omega_f)}^2\right] \\
&+  \dt  \sum_{n=m+1}^{N-1}  \frac{1}{\nu_f\alpha^2} \| \pdt s_2^{n+1}\|_{L^2(\Omega_f)}^2 +\dt  \sum_{n=m}^{N-1} \Big( \frac{\nu_f}{\alpha^2} \| \nabla s_2^{n+1}\|_{L^2(\Omega_f)}^2 + \frac{1}{\alpha}\| s_2^{n+1} \|_{L^2(\Omega_f)}^2 \Big)  \\
&+ \dt  \sum_{n=m}^{N-1} \Big(\frac{1}{\nu_s} \| s_1^{n+1}+s_2^{n+1}\|_{L^2(\Sigma)}^2+\frac{1}{\nu_f}\| s_2^{n+1}\|_{L^2(\Sigma)}^2 \Big)+ \frac{1}{\alpha^2}\|s_2^N\|_{L^2(\Omega_f)}^2.
\end{alignat*}

\begin{theorem}\label{stabilitythm}
Let $\tw, \tu$ solve \eqref{seq1} and \eqref{seq2}  and assuming that $\eps_2^{n+1} \in V_f$ then the following estimate holds
\begin{alignat*}{1}
Z^N(\tw, \tu, \tla)+ \sum_{n=0}^{N-1} S^{n+1}(\tw, \tu, \tla) \le C \Xi_0^N(b_1, b_2, \eps_1, \eps_2).
\end{alignat*}
\end{theorem}
\begin{proof}
 Using Lemma \ref{errorLemma1} and taking the sum we get 
\begin{alignat*}{1}
Z^N(\tw, \tu, \tla)+ \sum_{n=0}^{N-1} S^{n+1}(\tw, \tu, \tla) =  Z^0(\tw, \tu, \tla) + \dt \sum_{n=0}^{N-1} F^{n+1}(\tw, \tu) + \frac{\dt}{\alpha} \sum_{n=0}^{N-1} \bl \eps_2^{n+1}, \tla^{n+1} \br.
\end{alignat*}

Using the Poincar{\'e} and trace inequalities, we easily obtain the following bound
\begin{alignat*}{1}
    \dt \sum_{n=0}^{N-1} F^{n+1}(\tw, \tu) \le & \frac{1}{8} \sum_{n=0}^{N-1} S^{n+1}(\tw, \tu, \tla) + C \dt \sum_{n=0}^{N-1} (\frac{1}{\nu_s}\| b_1^{n+1}\|_{L^2(\Omega_s)}^2+\frac{1}{\nu_f}\| b_2^{n+1}\|_{L^2(\Omega_f)}^2) \\
    & +C  \dt \sum_{n=0}^{N-1} \Big(\frac{1}{\nu_s} \| \eps_1^{n+1}+\eps_2^{n+1}\|_{L^2(\Sigma)}^2+\frac{1}{\nu_f}\| \eps_2^{n+1}\|_{L^2(\Sigma)}^2\Big).
\end{alignat*}

Using Lemma \ref{lemmasummation} and the Poincar{\'e} inequality we have 
\begin{alignat*}{1}
     \frac{\dt}{\alpha} \sum_{n=0}^{N-1} \bl \eps_2^{n+1}, \tla^{n+1} \br \le  & \frac{1}{4} \| \tu^N\|_{L^2(\Omega_f)}^2+\frac{1}{8} \sum_{n=0}^{N-1} S^{n+1}(\tw, \tu, \tla)  \\ 
     &+  C  \dt \sum_{n=1}^{N-1}\frac{1}{\nu_f\alpha^2} \| \pdt \eps_2^{n+1}\|_{L^2(\Omega_f)}^2 + C\dt \sum_{n=0}^{N-1} \Big(\frac{\nu_f}{\alpha^2} \| \nabla \eps_2^{n+1}\|_{L^2(\Omega_f)}^2 + \frac{1}{\alpha}\|\eps_2^{n+1} \|_{L^2(\Omega_f)}^2 \Big)  \\
     &+  C \frac{\dt}{\alpha} \sum_{n=0}^{N-1}  \| b_2^{n+1} \|_{L^2(\Omega_f)}^2  +C \frac{1}{\alpha^2}\|\eps_2^N\|_{L^2(\Omega_f)}^2.
\end{alignat*}
It follows from \eqref{eq:uwinitial} and the definition of $Z^0$ that $Z^0(\tw, \tu, \tla)=0$. 
We finish the proof by combining the above estimates. 
\end{proof}

The discrete time derivative of $\tu$ and $\tw$ solves \eqref{seq1} and \eqref{seq2} with $n\ge 1$ and discrete time derivatives of the data as the right-hand sides. Note that the initial conditions are $\pdt w^1$ in $\Omega_s$ and $\pdt u^1$ in $\Omega_f$. We then have the following corollary.

\begin{corollary}\label{cordt}
Let  $\tw, \tu$ solve \eqref{seq1} and \eqref{seq2}  and  assuming that $\pdt\eps_2^{n+1} \in V_f$ then
\begin{alignat*}{1}
&Z^N(\pdt \tw, \pdt \tu, \pdt \tla)+ \sum_{n=1}^{N-1} S^{n+1}(\pdt\tw, \pdt \tu, \pdt\tla) \\
\le & C Z^{1}(\pdt \tw, \pdt \tu, \pdt \tla)+ C \Xi_{1}^N(\pdt b_1, \pdt b_2, \pdt \eps_1, \pdt \eps_2)+C\|\pdt\eps_2^{2}\|_{L^2(\Omega_f)}^2. 
 \end{alignat*}
\end{corollary}

\begin{proof}
  The argument is identical to those of Lemma \ref{errorLemma1}, Lemma \ref{lemmasummation} and Theorem \ref{stabilitythm} except the initial conditions are not zeros. To be specific, using the relation \eqref{eq:errorlemma1} and sum from $n=1$ to $n=N-1$, we have
\begin{alignat*}{1}
&Z^N(\pdt\tw, \pdt\tu, \pdt\tla)+ \sum_{n=1}^{N-1} S^{n+1}(\pdt\tw, \pdt\tu, \pdt\tla) \\
&=  Z^1(\pdt\tw, \pdt\tu, \pdt\tla) + \dt \sum_{n=1}^{N-1} F^{n+1}(\pdt\tw, \pdt\tu) + \frac{\dt}{\alpha} \sum_{n=1}^{N-1} \bl \pdt\eps_2^{n+1}, \pdt\tla^{n+1} \br.
\end{alignat*}
The term involving $F^{n+1}$ can be estimated similarly as that of Theorem \ref{stabilitythm}. To estimate the last term, we follow the same procedure in Lemma \ref{lemmasummation} and Theorem \ref{stabilitythm}. We have an extra term due to the initial condition which can be estimated as
\begin{equation}
  -(\pdt\tu^1, \pdt\eps_2^{2})_f\le\frac14\|\pdt\tu^1\|_{L^2(\Omega_f)}^2+\|\pdt\eps_2^{2}\|_{L^2(\Omega_f)}^2,
\end{equation}
where the first term can be absorbed into $Z^1(\pdt\tw, \pdt\tu, \pdt\tla)$. We then immediately obtain the result.
\end{proof}
Similarly, the second-order discrete time derivative of $\tu$ and $\tw$ solves \eqref{seq1} and \eqref{seq2} with $n\ge2$ and the second-order discrete time derivatives of the data as the right-hand sides. Notice that the initial conditions are $\pdt^2 w^2$ in $\Omega_s$ and $\pdt^2 u^2$ in $\Omega_f$. We then have the following corollary.
\begin{corollary}\label{cordtt}
Let  $\tw, \tu$ solve \eqref{seq1} and \eqref{seq2}  and  assuming that $\pdt^2\eps_2^{n+1} \in V_f$ then
\begin{alignat*}{1}
&Z^N(\pdt^2 \tw, \pdt^2 \tu, \pdt^2 \tla)+ \sum_{n=2}^{N-1} S^{n+1}(\pdt^2\tw, \pdt^2 \tu, \pdt^2\tla) \\
\le& CZ^{2}(\pdt^2 \tw, \pdt^2 \tu, \pdt^2 \tla) + C \Xi_2^N(\pdt^2 b_1, \pdt^2 b_2, \pdt^2 \eps_1, \pdt^2 \eps_2)+C\|\pdt^2\eps_2^{3}\|_{L^2(\Omega_f)}^2. 
 \end{alignat*}
\end{corollary}

\begin{remark}\label{stabilityremark}
We can also analyze the problem \eqref{seq1}-\eqref{seq2} but replacing the homogeneous  Neumann boundary conditions with  homogeneous  Dirichlet boundary conditions. In other words, the problem with pure Dirichlet boundary conditions that is $\Gamma_D^i=\partial \Omega_i \backslash \Sigma $ for $i=s,f$. In this case exactly the same estimates hold. Now of course, we assume $\eps_2^{n} \in  V_f$ where $V_f$ has zero Dirichlet boundary conditions on  $\Gamma_D^f=\partial \Omega_f \backslash \Sigma $. 
\end{remark}

\subsection{$H^2$ stability in a special case}
In this section we prove $H^2$ estimates for $\tu$ in a special configuration.  Again, we assume that $\Omega=(0,1)^2$, and now assume $\Sigma$ is parallel to the $x$-axis (see Figure \ref{figure:domain}). We take advantage of the fact that the sides composing $ \Gamma_{N\!e}$ are perpendicular to the $x$-axis and, moreover, we also notice that $\Sigma$ is parallel to the $x$-axis. 

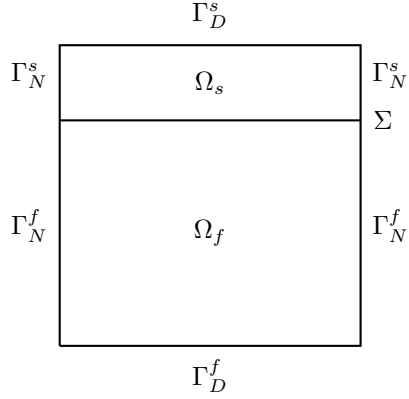
\begin{figure}[ht]
\centering
\begin{tikzpicture}
\draw[thick] plot coordinates {(0,0) (4,0) (4,4) (0,4) (0,0)};
\draw[thick] plot coordinates {(0,3) (4,3) };

 \node at (2,1.5) {$\Omega_f$};
 \node at (2,3.5) {$\Omega_s$};
 \node at (4.3,3) {$\Sigma$};
 \node at (2,-0.4) {$\Gamma^f_D$};
 \node at (2,4.4) {$\Gamma^s_D$};
 \node at (-0.4,1.6) {$ \Gamma_{N\!e}^f$};
 \node at (4.4,1.6) {$ \Gamma_{N\!e}^f$};
 \node at (-0.4,3.6) {$ \Gamma_{N\!e}^s$};
 \node at (4.4,3.6) {$ \Gamma_{N\!e}^s$};
\end{tikzpicture}
\caption{The domains $\Omega_f$ and $\Omega_s$ with horizontal interface $\Sigma$.} \label{figure:domain}
\end{figure}

Then, in this particular case we have, for $n\ge0$,
\begin{subequations}\label{seq1x}
\begin{alignat}{2}
\pdt \p_x \tw^{n+1}- \nu_s \Delta \partial_x \tw^{n+1}=&\p_x b_1^{n+1}, \quad && \text{ in }  \Omega_s, \\
\p_x \tw^{n+1}=&0 , \quad &&  \text{ on } \Gamma^s_D, \\
 \p_x \tw^{n+1}=&0, \quad &&  \text{ on }   \Gamma_{N\!e}^s,\\
 \alpha\p_x\tw^{n+1}+ \nu_s \p_{\bn_s} \p_x \tw^{n+1} = &  \alpha  \p_x \tu^{n} - \nu_f \p_{\bn_f} \p_x \tu^n+ \p_x \eps_1^{n+1}   \quad  && \text{ on }     \Sigma,\label{eq:rrdxw}
\end{alignat}
\end{subequations}

\begin{subequations}\label{seq2x}
\begin{alignat}{2}
\pdt \p_x \tu^{n+1}- \nu_f\Delta  \p_x \tu^{n+1}=&\p_x b_2^{n+1}, \quad && \text{ in } \Omega_f, \label{seq2.1x} \\
\p_x  \tu^{n+1}=&0 , \quad &&  \text{ on }  \Gamma^f_D, \\
\p_x  \tu^{n+1}=&0, \quad &&  \text{ on }   \Gamma_{N\!e}^f,\\
\alpha \p_x  \tu^{n+1}+ \nu_f \p_{\bn_f}  \p_x \tu^{n+1} = & \alpha  \p_x \tw^{n+1} + \nu_f \p_{\bn_f} \p_x  \tu^n + \p_x  \eps_2^{n+1}  \quad && \text{ on }     \Sigma,\label{eq:rrdxu}
\end{alignat}
\end{subequations}
with 
\begin{equation}
   \p_x w^0=\p_xu^0=0
\end{equation}
according to \eqref{eq:uwinitial}.

\begin{remark}
   The assumption that the interface is perpendicular to the sides and parallel to the $x$-direction is crucial. First, notice that, in this case, the outer normal direction $\bm{n}$ to two sides aligns with the $x$-direction, leading to $\partial_xw^{n+1}=\partial_{\bm{n}}w^{n+1}=0$ on $\Gamma_{N\!e}^s$ and $\partial_xu^{n+1}=\partial_{\bm{n}}u^{n+1}=0$ on $\Gamma_{N\!e}^f$ where we use \eqref{eq:nbw} and \eqref{eq:nbu}. Secondly, given \eqref{eq:rrw} and \eqref{eq:rru}, the conditions \eqref{eq:rrdxw} and \eqref{eq:rrdxu} are only valid in the case where the interface is perpendicular to the sides. 
  Unfortunately, we have not yet found a general method to prove the $H^2$ estimates for the non-horizontal case which it proved to be quite challenging. However, we believe this technique could be useful for the analysis of fluid-structure interaction (FSI) problems, given the similarity of the coupling conditions.
\end{remark}

We then get an immediate Corollary from Remark \ref{stabilityremark} and Theorem \ref{stabilitythm}.
\begin{corollary}\label{corx}
Suppose that $\Sigma$ is perpendicular to the two sides of $ \Gamma_{N\!e}$ as  in Figure \ref{figure:domain}. Let $\tw, \tu$ solve \eqref{seq1} and \eqref{seq2}  and  assuming that $\p_{x} \eps_2^{n} \in \{ v \in H^1(\Omega_f): v=0 \text{ on } \p \Omega_f \backslash \Sigma\}  $ on $ \Gamma_{N\!e}^f$ then the following estimate holds
\begin{alignat*}{1}
Z^N(\p_x \tw, \p_x \tu, \p_x \tla)+ \sum_{n=0}^{N-1} S^{n+1}(\p_x \tw, \p_x \tu, \p_x \tla) \le  C \Xi_0^N(\p_x b_1, \p_x b_2, \p_x \eps_1,  \p_x\eps_2).
\end{alignat*}
\end{corollary}

\begin{corollary}\label{corH2}
Under the hypothesis of Corollaries \ref{corx} and \ref{cordt} we have 
\begin{alignat*}{1}
    \dt \sum_{n=0}^{N-1} \nu_f \|D^2 \tu^{n+1}\|_{L^2(\Omega_f)}^2 \le  &  C \Big( \nu_f\Xi_1^N(\pdt b_1, \pdt b_2, \pdt \eps_1, \pdt \eps_2)+ \Xi_0^N(\p_x b_1, \p_x b_2, \p_x \eps_1,  \p_x\eps_2)\Big) \\
&+  C\nu_f\|\pdt\eps_2^{2}\|_{L^2(\Omega_f)}^2+ C \nu_fZ^{1}(\pdt \tw, \pdt \tu, \pdt \tla)\\
&+C  \dt \sum_{n=0}^{N-1} \nu_f \| b_2^{n+1}\|_{L^2(\Omega_f)}^2.
\end{alignat*}
\end{corollary}
\begin{proof}
From Corollary \ref{corx} we get
\begin{alignat*}{1}
    \dt \sum_{n=0}^{N-1} \nu_f \|\nabla \p_x \tu^{n+1}\|_{L^2(\Omega_f)}^2 \le    C \Xi_0^N(\p_x b_1, \p_x b_2, \p_x \eps_1,  \p_x\eps_2).
\end{alignat*}
Moreover,  using \eqref{seq2.1} and Corollary \ref{cordt}, we have
\begin{alignat*}{1}
    \dt \sum_{n=0}^{N-1} \nu_f \|\Delta \tu^{n+1}\|_{L^2(\Omega_f)}^2 = &  \dt \sum_{n=0}^{N-1} \nu_f \|\pdt \tu^{n+1}- b_2^{n+1}\|_{L^2(\Omega_f)}^2 \\
\le &  C \nu_fZ^{1}(\pdt \tw, \pdt \tu, \pdt \tla)+ C \nu_f\Xi_{1}^N(\pdt b_1, \pdt b_2, \pdt \eps_1, \pdt \eps_2)\\
&+C\nu_f\|\pdt\eps_2^{2}\|_{L^2(\Omega_f)}^2+ 2 \dt \sum_{n=0}^{N-1} \nu_f \| b_2^{n+1}\|_{L^2(\Omega_f)}^2.
\end{alignat*}
Finally, using the following estimate,
\begin{alignat*}{1}
    \dt \sum_{n=0}^{N-1} \nu_f \|\p_y^2 \tu^{n+1}\|_{L^2(\Omega_f)}^2 \le    2 \dt \sum_{n=0}^{N-1} \nu_f \Big( \|\p_x^2 \tu^{n+1}\|_{L^2(\Omega_f)}^2+  \|\Delta \tu^{n+1}\|_{L^2(\Omega_f)}^2 \Big),  
\end{alignat*}
and combining the above estimates we obtain the result.
\end{proof}

Similarly, if we subtract the consecutive levels of \eqref{seq1x}-\eqref{seq2x} and divide them by $\dt$, we notice that $\pdt\p_xu^{n+1}$ and $\pdt\p_xw^{n+1}$ satisfy the same equations as \eqref{seq1x}-\eqref{seq2x} from $n\ge1$, and with initial conditions $\pdt\p_xu^{1}$ in $\Omega_f$ and $\pdt\p_xw^{1}$ in $\Omega_s$. Similar to Corollary \ref{cordt}, we have the following, 
\begin{corollary}\label{cordx}
Suppose that $\Sigma$ is perpendicular to the two sides of $ \Gamma_{N\!e}$ as  in Figure \ref{figure:domain}. Let $\tw, \tu$ solve \eqref{seq1} and \eqref{seq2}  and  assuming that $\pdt\p_{x} \eps_2^{n} \in \{ v \in H^1(\Omega_f): v=0 \text{ on } \p \Omega_f \backslash \Sigma\}  $ on $ \Gamma_{N\!e}^f$ then the following estimate holds
\begin{alignat*}{1}
&Z^N(\pdt\p_x \tw, \pdt\p_x \tu, \pdt\p_x \tla)+ \sum_{n=1}^{N-1} S^{n+1}(\pdt\p_x \tw, \pdt\p_x \tu, \pdt\p_x \tla) \\
&\le  C Z^{1}(\pdt \p_x\tw, \pdt \p_x\tu, \pdt \p_x\tla)+C \Xi_1^N(\pdt\p_x b_1, \pdt\p_x b_2, \pdt\p_x \eps_1,  \pdt\p_x\eps_2)+C\|\pdt\p_x\eps_2^{2}\|_{L^2(\Omega_f)}^2.
\end{alignat*}
\end{corollary}

We then have the following corollary similar to Corollary \ref{corH2}.
\begin{corollary}\label{cordtx}
Under the hypothesis of Corollary \ref{cordx} we have 
\begin{alignat*}{1}
    &\dt \sum_{n=1}^{N-1} \nu_f \|D^2 (\pdt\tu^{n+1})\|_{L^2(\Omega_f)}^2 \\
    \le  & C \Big( \Xi_1^N(\pdt\p_x b_1, \pdt\p_x b_2, \pdt\p_x \eps_1,  \pdt\p_x\eps_2)+\nu_f \Xi_2^N(\pdt^2 b_1, \pdt^2 b_2, \pdt^2 \eps_1, \pdt^2 \eps_2)\Big) \\
&+ C\|\pdt\p_x\eps_2^{2}\|_{L^2(\Omega_f)}^2+ C Z^{1}(\pdt \p_x\tw, \pdt \p_x\tu, \pdt \p_x\tla)+C\nu_fZ^{2}(\pdt^2 \tw, \pdt^2 \tu, \pdt^2 \tla)\\
&+C\nu_f\|\pdt^2\eps_2^{3}\|_{L^2(\Omega_f)}^2+C  \dt \sum_{n=1}^{N-1} \nu_f \| \pdt b_2^{n+1}\|_{L^2(\Omega_f)}^2.
\end{alignat*}
\end{corollary}

\begin{proof}
  From Corollary \ref{cordx} we get
\begin{alignat*}{1}
    &\dt \sum_{n=1}^{N-1} \nu_f \|\nabla \p_x (\pdt\tu^{n+1})\|_{L^2(\Omega_f)}^2\\
     &\le     C Z^{1}(\pdt \p_x\tw, \pdt \p_x\tu, \pdt \p_x\tla)+C \Xi_1^N(\pdt\p_x b_1, \pdt\p_x b_2, \pdt\p_x \eps_1,  \pdt\p_x\eps_2)+C\|\pdt\p_x\eps_2^{2}\|_{L^2(\Omega_f)}^2.
\end{alignat*}

Moreover,  using the equation about $\pdt\p_x u^{n+1}$ and Corollary \ref{cordtt}, we have
\begin{alignat*}{1}
    \dt \sum_{n=1}^{N-1} \nu_f \|\Delta (\pdt\tu^{n+1})\|_{L^2(\Omega_f)}^2 = &  \dt \sum_{n=1}^{N-1} \nu_f \|\pdt^2\tu^{n+1}- \pdt b_2^{n+1}\|_{L^2(\Omega_f)}^2 \\
\le &   C\nu_fZ^{2}(\pdt^2 \tw, \pdt^2 \tu, \pdt^2 \tla) + C \nu_f\Xi_2^N(\pdt^2 b_1, \pdt^2 b_2, \pdt^2 \eps_1, \pdt^2 \eps_2)\\
&+C\nu_f\|\pdt^2\eps_2^{3}\|_{L^2(\Omega_f)}^2+ 2 \dt \sum_{n=1}^{N-1} \nu_f \| \pdt b_2^{n+1}\|_{L^2(\Omega_f)}^2. 
\end{alignat*}
Similar to Corollary \ref{corH2}, we finish the proof.
\end{proof}

\section{Error estimates of the Robin-Robin method}\label{sec:errorest}
In this section we apply the stability results of the previous sections to obtain error estimates of the Robin-Robin splitting method  \eqref{eq1}-\eqref{eq2} applied to \eqref{eq:ppinterface}. We use the following notation for the errors :  
\begin{alignat*}{1}
U^n:=& \uu^n-u^n, \quad W^n:=\ww^n-w^n, \quad \Lambda^{n}:=\lla^n-\lambda^n. 
\end{alignat*}

Then the error equations read, for $n\ge0$,

\begin{subequations}\label{erroreq1}
\begin{alignat}{2}
\pdt W^{n+1}- \nu_s \Delta W^{n+1}=&-h_1^{n+1}, \quad && \text{ in }  \Omega_s, \\
W^{n+1}=&0 , \quad &&  \text{ on } \Gamma^s_D, \\
\p_{\bn} W^{n+1}=&0, \quad &&  \text{ on }   \Gamma_{N\!e}^s,\\
\alpha  W^{n+1}+ \nu_s \p_{\bn_s} W^{n+1} = &  \alpha  U^{n} -\nu_f \p_{\bn_f} U^n + \alpha g_1^{n+1}-g_2^{n+1}  \quad  && \text{ on }     \Sigma.
\end{alignat}
\end{subequations}

\begin{subequations}\label{erroreq2}
\begin{alignat}{2}
\pdt U^{n+1}- \nu_f\Delta U^{n+1}=& -h_2^{n+1}, \quad && \text{ in } \Omega_f,   \\
U^{n+1}=&0 , \quad &&  \text{ on }  \Gamma^f_D, \\
\p_{\bn} U^{n+1}=&0, \quad &&  \text{ on }    \Gamma_{N\!e}^f,\\
\alpha U^{n+1}+ \nu_f \p_{\bn_f} U^{n+1} = & \alpha W^{n+1} + \nu_f \p_{\bn_f} U^n+ g_2^{n+1}   \quad && \text{ on }     \Sigma.
\end{alignat}
\end{subequations}
where 
\begin{alignat*}{2}
h_1^{n+1} &:= \pt \ww^{n+1} - \pdt \ww^{n+1}, \quad && g_1^{n+1} := \uu^{n+1}-\uu^n, \\
h_2^{n+1} &:= \pt \uu^{n+1} - \pdt \uu^{n+1}, \quad && g_2^{n+1} := \lla^{n+1} - \lla^n.
\end{alignat*}

The equations \eqref{erroreq1}-\eqref{erroreq2} are well-defined and we also have 
\begin{equation}\label{eq:wuinit}
W^0=U^0=0.   
\end{equation}

  The following assumptions are useful in the analysis (cf. \cite{aggul2018defect}):
\begin{assumption}\label{assump:initial}
\begin{equation}
 \|\partial_t\uu(0)\|_{L^2(\Sigma)}\le C(\dt)^\frac12,\quad\|\partial_t\lla(0)\|_{L^2(\Sigma)}\le C(\dt)^\frac12.
\end{equation}
\end{assumption}

\begin{assumption}\label{assump:initial1}
\begin{alignat}{2}
&\|\partial_t\uu(0)\|_{L^2(\Sigma)}\le C(\dt)^\frac32,&&\quad
\|\partial_t\lla(0)\|_{L^2(\Sigma)}\le C(\dt)^\frac32,\\
&\|\partial^2_t\uu(0)\|_{L^2(\Omega_f)}\le C\dt,&&\quad
\|\partial^2_t\ww(0)\|_{L^2(\Omega_s)}\le C\dt,\\
&\|\partial^2_t\uu(0)\|_{L^2(\Sigma)}\le C(\dt)^\frac12,&&\quad
\|\partial^2_t\lla(0)\|_{L^2(\Sigma)}\le C(\dt)^\frac12.
\end{alignat}
\end{assumption}

We then extend $\lla$ to $\Omega_f$ in a natural way. We let $\tilde{\lla}= \phi \nu_f\nabla \uu \cdot \bn_f$ where $\phi$ is a function that is one on $\Sigma$ and vanishes on $\Gamma_D^f$. Then, we define  $\tilde{g}_2^{n+1}=  \tilde{\lla}^{n+1} - \tilde{\lla}^n$. By construction $\tilde{g}_2^{n+1} \in V_f$ and $\tilde{g}_2^n= g_2^{n}$  on $\Sigma$.
We immediately get the following result if we apply Theorem \ref{stabilitythm}, Corollary \ref{cordt} and Corollary \ref{cordtt}. 
\begin{corollary}\label{errorestimate}
Let $\uu, \ww$ solve \eqref{eq:ppinterface} and $w, u$ solve \eqref{eq1} and \eqref{eq2}  then
\begin{alignat}{1}\label{eq:WUerror}
Z^N(W, U, \Lambda)+ \sum_{n=0}^{N-1} S^{n+1}(W, U, \Lambda) \le   C \Xi_0^N(h_1, h_2, \alpha g_1-g_2,\tilde{g}_2),
\end{alignat}
\begin{equation}\label{eq:WUdterror}
  \begin{aligned}
& Z^N(\pdt W, \pdt U, \pdt \Lambda)+ \sum_{n=1}^{N-1} S^{n+1}(\pdt W, \pdt U, \pdt \Lambda)  \\
&\le   CZ^1(\pdt W, \pdt U, \pdt \Lambda)+ C \Xi_1^N(\pdt h_1, \pdt h_2, \alpha \pdt g_1-\pdt g_2, \pdt \tilde{g}_2)+C\|\pdt \tilde{g}_2^{2}\|_{L^2(\Omega_f)}^2,
\end{aligned}
\end{equation}
and
\begin{equation}\label{eq:WUdt2error}
  \begin{aligned}
& Z^N(\pdt^2 W, \pdt^2 U, \pdt^2 \Lambda)+ \sum_{n=2}^{N-1} S^{n+1}(\pdt^2 W, \pdt^2 U, \pdt^2 \Lambda)  \\
&\le  C Z^2(\pdt^2 W, \pdt^2 U, \pdt^2 \Lambda)+ C \Xi_2^N(\pdt^2 h_1, \pdt^2 h_2, \alpha \pdt^2 g_1-\pdt^2 g_2, \pdt^2 \tilde{g}_2)+C\|\pdt^2\tilde{g}_2^{3}\|_{L^2(\Omega_f)}^2.
\end{aligned}
\end{equation}
\end{corollary}
Then, it is quite straightforward to get a convergence rate by estimating the right-hand sides. The proof of the following Corollaries can be found in Appendix \ref{appdix:initcoro} and Appendix \ref{appdix:fcoro}.

\begin{corollary}\label{coro:initzs}
  Let $\uu, \ww$ solve \eqref{eq:ppinterface} and $w, u$ solve \eqref{eq1} and \eqref{eq2}. Under Assumption \ref{assump:initial}, we have
\begin{alignat}{1}\label{eq:initz1}
Z^1(\pdt W, \pdt U, \pdt \Lambda)\le C(\dt)^2\mathsf{Y}_1.
\end{alignat}
and, under Assumption \ref{assump:initial1}, we have
\begin{alignat}{1}\label{eq:initz2}
Z^2(\pdt^2 W, \pdt^2 U, \pdt^2 \Lambda)\le C(\dt)^2\mathsf{Y}_2,
\end{alignat}
where $\mathsf{Y}_1$ and $\mathsf{Y}_2$ are defined in Appendix \ref{appdix:initcoro}.
\end{corollary}

\begin{corollary}\label{Corrate}
Let $\uu, \ww$ solve \eqref{eq:ppinterface} and $w, u$ solve \eqref{eq1} and \eqref{eq2}  then
\begin{alignat}{1}\label{eq:errorW}
Z^N(W, U, \Lambda)+ \sum_{n=0}^{N-1} S^{n+1}(W, U, \Lambda) \le  C (\dt)^2 \mathsf{Y},
\end{alignat}
\begin{alignat}{1}\label{eq:errordW}
& Z^N(\pdt W, \pdt U, \pdt \Lambda)+ \sum_{n=1}^{N-1} S^{n+1}(\pdt W, \pdt U, \pdt \Lambda) \le C (\dt)^2 (\mathcal{Y}+\mathsf{Y}_1),
\end{alignat}
and
\begin{alignat}{1}\label{eq:errorddW}
& Z^N(\pdt^2 W, \pdt^2 U, \pdt^2 \Lambda)+ \sum_{n=2}^{N-1} S^{n+1}(\pdt^2 W, \pdt^2 U, \pdt^2 \Lambda) \le C (\dt)^2 (\mathfrak{Y}+\mathsf{Y}_2),
\end{alignat}
where $\mathsf{Y}$ is defined as
\begin{equation}
  \begin{aligned}
\mathsf{Y}:=&  \frac{1}{\nu_s}\|\pt^2 \ww\|_{L^2(0,T;L^2(\Omega_s))}^2 + (\frac{1}{\nu_f}+\frac{1}{\alpha})\|\pt^2 \uu\|_{L^2(0,T;L^2(\Omega_f))}^2 \\
& + (\frac{\nu_f}{\alpha^2}+\frac{\nu_f^2}{\alpha})  \|\pt \uu\|_{L^2(0,T;H^1(\Omega_f))}^2+ \frac{(\nu_f)^3}{\alpha^2}  \|\pt \uu\|_{L^2(0,T;H^2(\Omega_f))}^2\\
& +\frac{\alpha^2}{\nu_s}\|\pt \uu\|_{L^2(0,T;L^2(\Sigma))}^2 + \frac{1}{\nu_f}\|\pt \lla\|_{L^2(0,T;L^2(\Sigma))}^2+\frac{\nu_f^2}{\alpha^2}  \|\pt \uu\|_{L^\infty(0,T;H^1(\Omega_f))}^2,
\end{aligned}
\end{equation}
$\mathcal{Y}$ is defined as
\begin{equation}\label{eq:caly}
  \begin{aligned}
\mathcal{Y}:=&  \frac{1}{\nu_s}\|\pt^3 \ww\|_{L^2(0,T;L^2(\Omega_s))}^2 + (\frac{1}{\nu_f}+\frac{1}{\alpha})\|\pt^3 \uu\|_{L^2(0,T;L^2(\Omega_f))}^2 \\
& + (\frac{\nu_f}{\alpha^2}+\frac{\nu_f^2}{\alpha})  \|\pt^2 \uu\|_{L^2(0,T;H^1(\Omega_f))}^2+ \frac{(\nu_f)^3}{\alpha^2}  \|\pt^2 \uu\|_{L^2(0,T;H^2(\Omega_f))}^2\\
& +\frac{\alpha^2}{\nu_s}\|\pt^2 \uu\|_{L^2(0,T;L^2(\Sigma))}^2 + \frac{1}{\nu_f}\|\pt^2 \lla\|_{L^2(0,T;L^2(\Sigma))}^2+(\frac{\nu_f^2}{\alpha^2}+\nu_f^2)  \|\pt^2 \uu\|_{L^\infty(0,T;H^1(\Omega_f))}^2,
\end{aligned}
\end{equation}
and $\mathfrak{Y}$ is defined as
\begin{equation}\label{eq:frakY}
\begin{aligned}
 \mathfrak{Y}:=&  \frac{1}{\nu_s}\|\pt^4 \ww\|_{L^2(0,T;L^2(\Omega_s))}^2 + (\frac{1}{\nu_f}+\frac{1}{\alpha})\|\pt^4 \uu\|_{L^2(0,T;L^2(\Omega_f))}^2 \\
& + (\frac{\nu_f}{\alpha^2}+\frac{\nu_f^2}{\alpha})  \|\pt^3 \uu\|_{L^2(0,T;H^1(\Omega_f))}^2+ \frac{(\nu_f)^3}{\alpha^2}  \|\pt^3 \uu\|_{L^2(0,T;H^2(\Omega_f))}^2\\
& +\frac{\alpha^2}{\nu_s}\|\pt^3 \uu\|_{L^2(0,T;L^2(\Sigma))}^2 + \frac{1}{\nu_f}\|\pt^3 \lla\|_{L^2(0,T;L^2(\Sigma))}^2+\frac{\nu_f^2}{\alpha^2}  \|\pt^3 \uu\|_{L^\infty(0,T;H^1(\Omega_f))}^2.
\end{aligned}  
\end{equation}
\end{corollary}

\subsection{$H^2$ error estimates in a special case}
Here we assume that we have the configuration as in Figure \ref{figure:neumann}. Then, we see that $\partial_x \tilde{g_2}^{n+1}$ vanishes on $ \Gamma_{N\!e}^f$ and $\Gamma_D^f$ and hence belongs to $V_f$. Hence, Corollary \ref{corH2} gives the following corollary. 

\begin{corollary}\label{H2error}
 Suppose that $\Sigma$ is perpendicular to the two sides of $\Gamma_{N\!e}$ as  in Figure \ref{figure:domain}. Let $\uu, \ww$ solve \eqref{eq:ppinterface} and $w, u$ solve \eqref{eq1} and \eqref{eq2}  then
\begin{alignat*}{1}
    \dt \sum_{n=0}^{N-1} \nu_f \|D^2 (U^{n+1})\|_{L^2(\Omega_f)}^2 \le 
&  C\nu_f\Xi_1^N(\pdt h_1, \pdt h_2, \pdt (\alpha g_1-g_2), \pdt \tilde{g}_2) \\
& + C \Xi_0^N(\p_x  h_1, \p_x h_2, \p_x (\alpha g_1-g_2),   \p_x \tilde{g}_2) \\
&+C\nu_f\|\pdt\tilde{g}_2^{2}\|_{L^2(\Omega_f)}^2+ C \nu_fZ^{1}(\pdt W, \pdt U, \pdt \Lambda)\\
&+ C  \dt \sum_{n=0}^{N-1} \nu_f \|h_2^{n+1}\|_{L^2(\Omega_f)}^2.
\end{alignat*}
\end{corollary}
Since $\pdt U$ and $\pdt W$ satisfy the same equations as $U, W$ with time difference right-hand sides, we have 
\begin{corollary}\label{H2errordt}
 Suppose that $\Sigma$ is perpendicular to the two sides of $\Gamma_{N\!e}$ as  in Figure \ref{figure:domain}. Let $\uu, \ww$ solve \eqref{eq:ppinterface} and $w, u$ solve \eqref{eq1} and \eqref{eq2}, then
 \begin{equation}\label{eq:h2duesti}
\begin{aligned}
    \dt \sum_{n=0}^{N-1} \nu_f \|D^2 (\pdt U^{n+1})\|_{L^2(\Omega_f)}^2 \le  
&  C \Big( \Xi_1^N(\pdt\p_x h_1, \pdt\p_x h_2, \pdt\p_x (\alpha g_1-g_2),  \pdt\p_x\tilde{g}_2)\\
&+\nu_f \Xi_2^N(\pdt^2 h_1, \pdt^2 h_2, \pdt^2 (\alpha g_1-g_2), \pdt^2 \tilde{g}_2)\Big) \\
&+  C\|\pdt\p_x\tilde{g}_2^{2}\|_{L^2(\Omega_f)}^2+ C Z^{1}(\pdt \p_xW, \pdt \p_xU, \pdt \p_x\Lla)\\
&+C\nu_fZ^{2}(\pdt^2 W, \pdt^2 U, \pdt^2 \Lla)\\
&+C\nu_f\|\pdt^2\tilde{g}_2^{3}\|_{L^2(\Omega_f)}^2+C  \dt \sum_{n=1}^{N-1} \nu_f \| \pdt h_2^{n+1}\|_{L^2(\Omega_f)}^2.
\end{aligned}
\end{equation}
\end{corollary}

Then, it is straight-forward to obtain the convergence rate by estimate the right-hand sides. See Appendix \ref{appdix:fcoro1} for a proof for the following Corollary. Note that we need the following additional assumptions.

 \begin{assumption}\label{assump:initial2}
 \begin{equation}
     \|\partial_t\p_x\uu(0)\|_{L^2(\Sigma)}\le C(\dt)^\frac12,\quad\|\partial_t\p_x\lla(0)\|_{L^2(\Sigma)}\le C(\dt)^\frac12.
 \end{equation}
\end{assumption}

\begin{corollary}\label{H2errordtexplcit}
 Suppose that $\Sigma$ is perpendicular to the two sides of $\Gamma_{N\!e}$ as  in Figure \ref{figure:domain}. Let $\uu, \ww$ solve \eqref{eq:ppinterface} and $w, u$ solve \eqref{eq1} and \eqref{eq2}  then, under Assumptions \ref{assump:initial1} and \ref{assump:initial2}, we have
\begin{alignat*}{1}
    \dt \sum_{n=1}^{N-1}& \nu_f \|D^2 (\pdt U^{n+1})\|_{L^2(\Omega_f)}^2 \\
    &\le  C(\dt)^2 \Big(\nu_f\mathfrak{Y}+\mathbb{Y}+\nu_f\mathsf{Y}_2+\mathsf{Y}_3\\
    &\ \ +\nu_f^2\|\pt^2\uu\|^2_{L^\infty((0,T),H^2(\Omega_f))}+\nu_f^3\|\pt^3\uu\|^2_{L^\infty((0,T),H^1(\Omega_f))}+\nu_f\|\pt^3\uu\|^2_{L^2((0,T),L^2(\Omega_f))}\Big)
\end{alignat*}
where $\mathfrak{Y}$ is defined in Corollary \ref{Corrate}
and $\mathbb{Y}$ is defined as
\begin{equation}\label{eq:bbY}
\begin{aligned}
 \mathbb{Y}:=&  \frac{1}{\nu_s}\|\p_x\pt^3 \ww\|_{L^2(0,T;L^2(\Omega_s))}^2 + (\frac{1}{\nu_f}+\frac{1}{\alpha})\|\p_x\pt^3 \uu\|_{L^2(0,T;L^2(\Omega_f))}^2 \\
& + (\frac{\nu_f}{\alpha^2}+\frac{\nu_f^2}{\alpha})  \|\p_x\pt^2 \uu\|_{L^2(0,T;H^1(\Omega_f))}^2+ \frac{(\nu_f)^3}{\alpha^2}  \|\p_x\pt^2 \uu\|_{L^2(0,T;H^2(\Omega_f))}^2\\
& +\frac{\alpha^2}{\nu_s}\|\p_x\pt^2 \uu\|_{L^2(0,T;L^2(\Sigma))}^2 + \frac{1}{\nu_f}\|\p_x\pt^2 \lla\|_{L^2(0,T;L^2(\Sigma))}^2+\frac{\nu_f^2}{\alpha^2}  \|\p_x\pt^2 \uu\|_{L^\infty(0,T;H^1(\Omega_f))}^2
\end{aligned}  
\end{equation}
and $\mathsf{Y}_3$ is defined as,
  \begin{equation*}
\begin{aligned}
    \mathsf{Y}_3&=\max_{0 \le  s \le t_1} \| \partial_t^2\p_x \ww(s)\|_{L^2(\Omega_s)}^2+\max_{0 \le  s \le t_1} \| \partial_t^2 \p_x\uu(s)\|_{L^2(\Omega_f)}^2\\
    &+\frac{\alpha^2}{\nu_s}(1+\dt\max_{0 \le  s \le t_1} \| \partial_t^2 \p_x\uu(s)\|_{L^2(\Sigma)}^2)\\
    &+\frac{1}{\alpha}(1+\dt \max_{0 \le  s \le t_1} \| \partial_t^2 \p_x\lla(s)\|_{L^2(\Sigma)}^2).
\end{aligned}
\end{equation*}
\end{corollary}

\section{Numerical Experiments}\label{sec:numerics}

In this section we provide numerical experiments that agree with our theoretical results. We let 
  \begin{alignat*}{3}
    e_u&= \|U^N\|_{L^2(\Omega_f)},\quad e_{1,u}=  \|U^N-U^{N-1}\|_{L^2(\Omega_f)},\\
    e_{2,u}&=  \|(U^N-U^{N-1})-(U^{N-1}-U^{N-2})\|_{L^2(\Omega_f)},\\
    e_{\lambda}&=\|\Lambda^N\|_{L^2(\Sigma)},\quad e_{1,\lambda}=\|\Lambda^N-\Lambda^{N-1}\|_{L^2(\Sigma)}\\
    e_{1,u,2}&=\|U^N-U^{N-1}\|_{H^2(\Omega_f)}.
  \end{alignat*}

All numerical experiments are performed using FEniCS and multiphenics \cite{alnaes2015fenics,multiphenics}. Although we only analyze the semi-discrete method, here we present the results for a fully discrete method where we use the piecewise linear finite element method for the spatial discretization, except for the computation of $e_{1,u,2}$, where we use the piecewise quadratic finite element method because piecewise linear function do not approximation a function well in $H^2$-norm. In addition, we also present convergence rates for the Lagrange multiplier. 

\begin{example}\label{ex:hrztalinter}
  We consider the domain $\Omega=(0,1)^2$, $\Omega_f =(0,1)\times (0,0.75)$ and $\Omega_s=(0,1) \times (0.75,1)$.  See Figure \ref{figure:domain} for an illustration. We take $\nu_f=1=\nu_s$ and take the solution of \eqref{var} to be
\begin{equation*}
    \ww=\uu= e^{-2\pi^2 t} \cos(\pi x_1) \sin(\pi x_2).
\end{equation*}
\end{example}

We take $h=\dt$, $T=0.25$ and $\alpha=4$ where $h$ is the mesh size of the triangulation.

\begin{table}[h]
\begin{center}

\begin{tabular}{|c||c|c|c|c|c|c|}
\hline
$\Dt$ &  $e_u$ &  rates & $e_{1,u}$ &  rates & $e_{2,u}$ & rates \\
\hline
$(1/2)^2$&7.65e-02&-&7.73e-02&-&6.57e+01&-\\
\hline
$(1/2)^3$&3.72e-02&1.04&5.87e-02&0.40&1.55e-01&8.73\\
\hline
$(1/2)^4$&1.74e-02&1.10&1.47e-02&1.99&7.91e-03&4.29\\
\hline
$(1/2)^5$&7.95e-03&1.13&3.58e-03&2.04&1.27e-03&2.64\\
\hline
$(1/2)^6$&3.52e-03&1.17&8.41e-04&2.09&1.75e-04&2.85\\
\hline
$(1/2)^7$&1.62e-03&1.12&1.96e-04&2.10&2.15e-05&3.03\\
\hline
$(1/2)^8$&7.70e-04&1.07&4.69e-05&2.06&2.62e-06&3.04\\
\hline
$(1/2)^9$&3.75e-04&1.04&1.14e-05&2.04&3.22e-07&3.02\\
\hline
\end{tabular}
\end{center}
\caption{Errors and convergence rates of $U^N$ for Example \ref{ex:hrztalinter}}\label{fig:conv1}
\end{table}

\begin{table}[h]
\begin{center}
\begin{tabular}{|c||c|c|c|c|c|c|}
\hline
$\Dt$ &  $e_{\lambda}$ &  rates & $e_{1,\lambda}$ &  rates & $e_{1,u,2}$ & rates \\
\hline
$(1/2)^2$&2.55e-01&-&2.55e-01&-&1.33e+01&-\\
\hline
$(1/2)^3$&9.73e-02&1.39&2.41e-01&0.08&1.11e+01&0.26\\
\hline
$(1/2)^4$&3.06e-02&1.67&4.41e-02&2.45&2.61e-01&2.08\\
\hline
$(1/2)^5$&1.62e-02&0.92&6.69e-03&2.72&5.76e-02&2.18\\
\hline
$(1/2)^6$&9.20e-03&0.82&2.06e-03&1.70&1.40e-02&2.04\\
\hline
$(1/2)^7$&4.55e-03&1.01&5.28e-04&1.97&3.32e-03&2.07\\
\hline
$(1/2)^8$&2.23e-03&1.03&1.31e-04&2.01&8.03e-04&2.05\\
\hline
$(1/2)^9$&1.10e-03&1.02&3.26e-05&2.01&1.97e-04&2.03\\
\hline

\end{tabular}
\end{center}
\caption{Error and convergence rates of $\Lambda^N$ for Example \ref{ex:hrztalinter}}\label{fig:conv2}
\end{table}

As we can see from Tables \ref{fig:conv1}-\ref{fig:conv2}, the $L^2$ error at the final step, $e_u$ is of order $(\Delta t)$ whereas the difference of two consecutive errors, $e_{1,u}$ is of order $(\Delta t)^2$ and the second difference, $e_{2,u}$ is of order $(\Delta t)^3$. The $L^2$ error of the Lagrange multiplier at the final time, $e_{\lambda}$ is of order $\Delta t$ and the difference $e_{1,\lambda}$, is of order $(\Delta t)^2$. It is also clear that the $H^2$ error of the difference of $U^N$, $e_{1,u,2}$ is of order $(\dt)^2$ as we proved in Corollary \ref{H2errordtexplcit}.

\begin{example}\label{ex:nonhoriztal}
  In this example, we test our algorithm for a non-horizontal interface problem. We consider the domain $\Omega=(0,1)^2$ and we let $\Sigma$ be defined as the straight line connecting $(0,0.25)$ and $(1,0.75)$. We then define $\Omega_s$ as the region above $\Sigma$ and $\Omega_f$ as the region below $\Sigma$. We take $\nu_f=1=\nu_s$ and take the solution of \eqref{var} to be
\begin{equation*}
    \ww=\uu= e^{-2\pi^2 t} \cos(\pi x_1) \sin(\pi x_2).
\end{equation*}
Other parameters are identical to those of Example \ref{ex:hrztalinter}.
\end{example}

We report the convergence results in Tables \ref{fig:conv3}-\ref{fig:conv4}. We again observe expected convergence rates for both $U^N$ and $\Lambda^N$. It indicates that our methods also work for a more general interface problem.

\begin{table}[h]
\begin{center}

\begin{tabular}{|c||c|c|c|c|c|c|}
\hline
$\Dt$ &  $e_u$ &  rates & $e_{1,u}$ &  rates & $e_{2,u}$ & rates \\
\hline
$(1/2)^2$&8.05e-02&-&9.64e-02&-&4.87e+01&-\\
\hline
$(1/2)^3$&5.10e-02&0.66&4.23e-02&1.19&1.36e-01&8.48\\
\hline
$(1/2)^4$&2.46e-02&1.05&1.57e-02&1.43&4.94e-03&4.78\\
\hline
$(1/2)^5$&9.20e-03&1.42&4.07e-03&1.95&1.47e-03&1.74\\
\hline
$(1/2)^6$&3.52e-03&1.38&8.46e-04&2.27&1.82e-04&3.02\\
\hline
$(1/2)^7$&1.52e-03&1.22&1.86e-04&2.18&2.11e-05&3.11\\
\hline
$(1/2)^8$&7.00e-04&1.11&4.33e-05&2.10&2.49e-06&3.08\\
\hline
$(1/2)^9$&3.36e-04&1.06&1.04e-05&2.06&3.02e-07&3.04\\
\hline
\end{tabular}
\end{center}
\caption{Errors and convergence rates of $U^N$ for Example \ref{ex:nonhoriztal}}\label{fig:conv3}
\end{table}

\begin{table}[h]
\begin{center}
\begin{tabular}{|c||c|c|c|c|c|c|}
\hline
$\Dt$ &  $e_{\lambda}$ &  rates & $e_{1,\lambda}$ &  rates & $e_{1,u,2}$ & rates \\
\hline
$(1/2)^2$&8.51e-01&-&8.54e-01&-&2.04e+01&-\\
\hline
$(1/2)^3$&5.01e-01&0.76&3.81e-01&1.16&9.67e-01&1.08\\
\hline
$(1/2)^4$&1.94e-01&1.37&1.45e-01&1.39&3.66e-01&1.40\\
\hline
$(1/2)^5$&5.34e-02&1.86&2.44e-02&2.58&9.08e-02&2.01\\
\hline
$(1/2)^6$&1.84e-02&1.54&4.38e-03&2.48&1.86e-02&2.29\\
\hline
$(1/2)^7$&7.49e-03&1.30&9.10e-04&2.27&4.04e-03&2.20\\
\hline
$(1/2)^8$&3.39e-03&1.14&2.07e-04&2.13&9.31e-04&2.11\\
\hline
$(1/2)^9$&1.61e-03&1.07&4.95e-05&2.07&2.23e-04&2.06\\
\hline
\end{tabular}
\end{center}
\caption{Error and convergence rates of $\Lambda^N$ for Example \ref{ex:nonhoriztal}}\label{fig:conv4}
\end{table}

\section{Concluding Remarks}\label{sec:conclude}

We analyzed the Robin-Robin coupling methods \cite{burman2023loosely} for parabolic-parabolic interface problems and proved higher convergence rates in time for the first-order and second-order discrete time derivatives. We also prove $H^2$ estimates of the discrete time derivatives in a special case. All the estimates in this work are key ingredients in proving that a prediction correction method \cite{ppcorrection}  produces a $O((\dt)^2)$ convergence rate.

\appendix

\section{Sketch of Proof: Corollary \ref{coro:initzs}}\label{appdix:initcoro}

\begin{proof}[Proof of \eqref{eq:initz1} in Corollary \ref{coro:initzs}]
  It follows from \eqref{eq:errorlemma1} that
  \begin{alignat*}{1}\label{eq:z1}
Z^{1}(W, U, \Lambda)+S^{1}(W, U, \Lambda)&=Z^{0}(W, U, \Lambda) + \dt F^{1}(W, U)  +\frac{\dt}{\alpha} \bl \tilde{g}_2^1, \Lambda^1 \br\\
&=\dt F^{1}(W, U)  +\frac{\dt}{\alpha} \bl \tilde{g}_2^1, \Lambda^1 \br,
\end{alignat*}
where we use the fact that $Z^{0}(W, U, \Lambda)=0$. By the definition of $F^{1}$ in Lemma \ref{errorLemma1} and the proof in Lemma \ref{lemmasummation}, we have
\begin{alignat*}{1}
Z^{1}(W, U, \Lambda)+S^{1}(W, U, \Lambda)&=\dt\Big((-h_1^1,W^1)_s+(-h_2^1, U^1)_f+\bl\alpha g_1^1, W^1\br+\bl g_2^1, U^1\br\Big)
+\frac{\dt}{\alpha} \bl \tilde{g}_2^1, \Lambda^1 \br\\
&=S_1+S_2+\ldots+S_5.
\end{alignat*}
For $S_1$, we have
\begin{equation*}
\begin{aligned}
   \dt(-h_1^1,W^1)_s&\le \frac14\|W^1\|^2_{L^2(\Omega_s)}+C(\dt)^2\|h_1^1\|^2_{L^2(\Omega_s)}\\
   &\le\frac14\|W^1\|^2_{L^2(\Omega_s)}+C(\dt)^4\max_{0 \le  s \le t_1} \| \partial_t^2 \ww(s)\|_{L^2(\Omega_s)}^2.
\end{aligned}
\end{equation*}
The term $S_2$ can be estimated similarly: 
\begin{equation*}
\begin{aligned}
   \dt(-h_2^1,U^1)_f\le\frac14\|U^1\|^2_{L^2(\Omega_f)}+C(\dt)^4\max_{0 \le  s \le t_1} \| \partial_t^2 \uu(s)\|_{L^2(\Omega_f)}^2.
\end{aligned}
\end{equation*}
The term $S_3$ is estimated as follows where we use a trace inequality,
\begin{equation}
  \begin{aligned}
    \dt\bl\alpha g_1^1, W^1\br&\le \frac{\nu_s\dt}{2}\|\nabla W^1\|^2_{L^2(\Omega_s)}+C\frac{\alpha^2}{\nu_s}\dt\|g_1^1\|^2_{L^2(\Sigma)}\\
    &= \frac{\nu_s\dt}{2}\|\nabla W^1\|^2_{L^2(\Omega_s)}+C\frac{\alpha^2}{\nu_s}\dt\|\uu^{1}-\uu^0\|^2_{L^2(\Sigma)}.
  \end{aligned}
\end{equation}
Moreover, using 
\begin{equation*}
\uu^{1}-\uu^0= \dt \partial_t \uu(0) + \int_{0}^{t_1} \int_{0}^r \partial_t^2 \uu(s) ds dr.
\end{equation*}
we have 
\begin{equation}\label{eq:du0}
 \|  \uu^{1}-\uu^0\|_{L^2(\Sigma)}^2  \le C (\dt)^2 \| \partial_t \uu(0)\|_{L^2(\Sigma)}^2+C (\dt)^4 \max_{0 \le  s \le t_1} \| \partial_t^2 \uu(s)\|_{L^2(\Sigma)}^2.
\end{equation}
It follows from \eqref{eq:du0} that
\begin{equation*}
    \dt\bl\alpha g_1^1, W^1\br\le \frac{\nu_s\dt}{2}\|\nabla W^1\|^2_{L^2(\Omega_s)}+C\frac{\alpha^2}{\nu_s}\dt((\dt)^2 \| \partial_t \uu(0)\|_{L^2(\Sigma)}^2+(\dt)^4 \max_{0 \le  s \le t_1} \| \partial_t^2 \uu(s)\|_{L^2(\Sigma)}^2).
\end{equation*}
Similarly, we have the following estimate for $S_4$:
\begin{equation*}
    \dt\bl g_2^1, U^1\br\le \frac{\alpha\dt}{4}\|U^1\|^2_{L^2(\Sigma)}+C\frac{\dt}{\alpha}((\dt)^2 \| \partial_t \lla(0)\|_{L^2(\Sigma)}^2+(\dt)^4 \max_{0 \le  s \le t_1} \| \partial_t^2 \lla(s)\|_{L^2(\Sigma)}^2).
\end{equation*}
At last, for $S_5$, we have,
\begin{equation}
  \begin{aligned}
    \frac{\dt}{\alpha} \bl \tilde{g}_2^1, \Lambda^1 \br&\le \frac{\dt}{4\alpha}\|\Lambda^1\|^2_{L^2(\Sigma)}+C\frac{\dt}{\alpha}\|g_2^1\|^2_{L^2(\Sigma)}\\
    &\le \frac{\dt}{4\alpha}\|\Lambda^1\|^2_{L^2(\Sigma)}+C\frac{\dt}{\alpha}((\dt)^2 \| \partial_t \lla(0)\|_{L^2(\Sigma)}^2+ (\dt)^4 \max_{0 \le  s \le t_1} \| \partial_t^2 \lla(s)\|_{L^2(\Sigma)}^2).
  \end{aligned}
\end{equation}
Combining all the estimates above we obtain
\begin{equation}\label{eq:z1s1}
\begin{aligned}
    Z^{1}(W, U, &\Lambda)+S^{1}(W, U, \Lambda)\\
    \le& C(\dt)^4\max_{0 \le  s \le t_1} \| \partial_t^2 \ww(s)\|_{L^2(\Omega_s)}^2+C(\dt)^4\max_{0 \le  s \le t_1} \| \partial_t^2 \uu(s)\|_{L^2(\Omega_f)}^2\\
    &+C\frac{\alpha^2}{\nu_s}\dt((\dt)^2 \| \partial_t \uu(0)\|_{L^2(\Sigma)}^2+C (\dt)^4 \max_{0 \le  s \le t_1} \| \partial_t^2 \uu(s)\|_{L^2(\Sigma)}^2)\\
    &+C\frac{\dt}{\alpha}((\dt)^2 \| \partial_t \lla(0)\|_{L^2(\Sigma)}^2+C (\dt)^4 \max_{0 \le  s \le t_1} \| \partial_t^2 \lla(s)\|_{L^2(\Sigma)}^2).
\end{aligned}
\end{equation}
Therefore, it follows from Assumption \ref{assump:initial} that
\begin{equation}
  Z^{1}(W, U, \Lambda)\le C(\dt)^4\mathsf{Y}_1,
\end{equation}
where 
\begin{equation*}
\begin{aligned}
    \mathsf{Y}_1&=\max_{0 \le  s \le t_1} \| \partial_t^2 \ww(s)\|_{L^2(\Omega_s)}^2+\max_{0 \le  s \le t_1} \| \partial_t^2 \uu(s)\|_{L^2(\Omega_f)}^2\\
    &+\frac{\alpha^2}{\nu_s}(1+\dt\max_{0 \le  s \le t_1} \| \partial_t^2 \uu(s)\|_{L^2(\Sigma)}^2)\\
    &+\frac{1}{\alpha}(1+\dt \max_{0 \le  s \le t_1} \| \partial_t^2 \lla(s)\|_{L^2(\Sigma)}^2).
\end{aligned}
\end{equation*}
Notice that $Z^1(\pdt W, \pdt U, \pdt \Lambda)=\frac{1}{(\dt)^2}Z^1(W^1-W^0, U^1-U^0, \Lambda^1-\Lambda^0)=\frac{1}{(\dt)^2}Z^1(W^1, U^1, \Lambda^1)$ due to \eqref{eq:wuinit}. Hence we have
\begin{equation}
  Z^1(\pdt W, \pdt U, \pdt \Lambda)\le C(\dt)^2\mathsf{Y}_1.
\end{equation}
\end{proof}

\begin{proof}[Proof of \eqref{eq:initz2} in Corollary \ref{coro:initzs}]

Denote $\mathcal{U}^{n+1}=\pdt U^{n+1}$, $\mathcal{W}^{n+1}=\pdt W^{n+1}$, $\mathcal{L}^{n+1}=\pdt \Lambda^{n+1}$. First, we notice the following equation is valid.
\begin{subequations}\label{eq:stweakpdt2}
\begin{alignat}{2}
(\frac{\mathcal{W}^2-\mathcal{W}^1}{\dt}, z)_s+\nu_s(\nabla \mathcal{W}^{2}, \nabla z)_s + \alpha \bl \mathcal{W}^{2} - \mathcal{U}^1, z \br + \bl \mathcal{L}^1, z\br=& \mathsf{L}_1(z) , \quad && z \in V_s, \label{stweak1pdt2}\\
(\frac{\mathcal{U}^2-\mathcal{U}^1}{\dt}, v)_f+\nu_f(\nabla \mathcal{U}^2, \nabla v)_f- \bl \mathcal{L}^2, v\br=&\mathsf{L}_2(v), \quad &&v \in V_f,  \label{stweak2pdt2}\\
\bl \alpha(\mathcal{U}^2-\mathcal{W}^2)+(\mathcal{L}^2-\mathcal{L}^1), \mu \br=&\mathsf{L}_3(\mu), \quad && \mu \in V_g,  \label{stweak3pdt2}
\end{alignat}
\end{subequations}
where 
\begin{alignat*}{1}
\mathsf{L}_1(z) :=&  -(\pdt h_1^2, z)_s + \bl \alpha\pdt g_1^2-\pdt g_2^2, z \br  ,\\
\mathsf{L}_2(v):=&-(\pdt h_2^2,v)_f,\\
\mathsf{L}_3(\mu):=&  \bl \pdt g_2^2, \mu \br.
\end{alignat*}
Note that the following is also valid at $t_1$,
\begin{subequations}\label{eq:stweakt1}
\begin{alignat}{2}
(\frac{W^1}{\dt}, z)_s+\nu_s(\nabla W^{1}, \nabla z)_s+ \alpha \bl W^{1}, z \br=&-(h_1^1, z)_s + \bl \alpha g_1^1- g_2^1, z \br  , \quad && z \in V_s, \label{stweak1t1}\\
(\frac{U^1}{\dt}, v)_f+\nu_f(\nabla U^1, \nabla v)_f- \bl \Lambda^1, v\br=&-(h_2^1,v)_f, \quad &&v \in V_f,  \label{stweak2t1}\\
\bl \alpha(U^1-W^1)+\Lambda^1, \mu \br=& \bl g_2^1, \mu \br, \quad && \mu \in V_g,  \label{stweak3t1}
\end{alignat}
\end{subequations}
where we use the fact $W^0=U^0=\Lambda^0=0$.

Denote $\mathfrak{W}^2=\frac{W^2-2W^1}{\dt}$, $\mathfrak{U}^2=\frac{U^2-2U^1}{\dt}$, $\mathfrak{L}^2=\frac{\Lambda^2-2\Lambda^1}{\dt}$, $\mathfrak{W}^1=\frac{W^1}{\dt}$, $\mathfrak{U}^1=\frac{U^1}{\dt}$, $\mathfrak{L}^1=\frac{\Lambda^1}{\dt}$. Divide \eqref{eq:stweakt1} by $\dt$ and subtract it from \eqref{eq:stweakpdt2}, we have the following
\begin{subequations}\label{eq:stweakpdt22}
\begin{alignat}{2}
(\frac{\mathfrak{W}^2-\mathfrak{W}^1}{\dt}, z)_s+\nu_s(\nabla \mathfrak{W}^{2}, \nabla z)_s + \alpha \bl \mathfrak{W}^{2} - \mathfrak{U}^1, z \br + \bl \mathfrak{L}^1, z\br=& \mathbb{L}_1(z) , \quad && z \in V_s, \label{stweak1pdt22}\\
(\frac{\mathfrak{U}^2-\mathfrak{U}^1}{\dt}, v)_f+\nu_f(\nabla \mathfrak{U}^2, \nabla v)_f- \bl \mathfrak{L}^2, v\br=&\mathbb{L}_2(v), \quad &&v \in V_f,  \label{stweak2pdt22}\\
\bl \alpha(\mathfrak{U}^2-\mathfrak{W}^2)+(\mathfrak{L}^2-\mathfrak{L}^1), \mu \br=&\mathbb{L}_3(\mu), \quad && \mu \in V_g,  \label{stweak3pdt22}
\end{alignat}
\end{subequations}
where 
\begin{alignat*}{1}
\mathbb{L}_1(z) :=&  -(\frac{h_1^2-2h_1^1}{\dt}, z)_s + \bl \alpha\frac{g_1^2-2g_1^1}{\dt}-\frac{g_2^2-2g_2^1}{\dt}, z \br  ,\\
\mathbb{L}_2(v):=&-(\frac{h_2^2-2h_2^1}{\dt},v)_f,\\
\mathbb{L}_3(\mu):=&  \bl \frac{g_2^2-2g_2^1}{\dt}, \mu \br.
\end{alignat*}
It follows from Lemma \ref{errorLemma1} that the following relation holds:
\begin{equation*}
  Z^2(\mathfrak{W},\mathfrak{U},\mathfrak{L})+S^2(\mathfrak{W},\mathfrak{U},\mathfrak{L})=Z^1(\mathfrak{W},\mathfrak{U},\mathfrak{L})+\dt F^{2}(\mathfrak{W},\mathfrak{U})+\frac{\dt}{\alpha}\bl\frac{g_2^2-2g_2^1}{\dt},\mathfrak{L}^2\br.
\end{equation*}
Notice that $Z^2(\mathfrak{W},\mathfrak{U},\mathfrak{L})=(\dt)^2 Z^2(\pdt^2W,\pdt^2U,\pdt^2\Lambda)$ and $Z^1(\mathfrak{W},\mathfrak{U},\mathfrak{L})=Z^1(\pdt W,\pdt U,\pdt \Lambda)$. Therefore, we have
\begin{equation}\label{eq:zpdt2}
\begin{aligned}
   &(\dt)^2Z^2(\pdt^2W,\pdt^2U,\pdt^2\Lambda)+S^2(\mathfrak{W},\mathfrak{U},\mathfrak{L})\\
   &=Z^1(\pdt W,\pdt U,\pdt \Lambda)+\dt F^{2}(\mathfrak{W},\mathfrak{U})+\frac{\dt}{\alpha}\bl\frac{g_2^2-2g_2^1}{\dt},\mathfrak{L}^2\br.
\end{aligned}
\end{equation}
Now we estimate the right-hand side of \eqref{eq:zpdt2} term by term.

For $Z^1(\pdt W,\pdt U,\pdt \Lambda)$, it follows from \eqref{eq:z1s1} that
\begin{equation}\label{eq:z1ddt}
  \begin{aligned}
    Z^1(\pdt W,\pdt U,\pdt \Lambda)\le& C(\dt)^2\max_{0 \le  s \le t_1} \| \partial_t^2 \ww(s)\|_{L^2(\Omega_s)}^2+C(\dt)^2\max_{0 \le  s \le t_1} \| \partial_t^2 \uu(s)\|_{L^2(\Omega_f)}^2\\
    &+C\frac{\alpha^2}{\nu_s}\dt( \| \partial_t \uu(0)\|_{L^2(\Sigma)}^2+C (\dt)^2 \max_{0 \le  s \le t_1} \| \partial_t^2 \uu(s)\|_{L^2(\Sigma)}^2)\\
    &+C\frac{\dt}{\alpha}(\| \partial_t \lla(0)\|_{L^2(\Sigma)}^2+C (\dt)^2 \max_{0 \le  s \le t_1} \| \partial_t^2 \lla(s)\|_{L^2(\Sigma)}^2).
\end{aligned}
\end{equation}
By the Mean Value Theorem, we have
\begin{equation}\label{eq:mvt}
  \partial_t^2 \ww(s)=\partial_t^2\ww(0)+s\partial^3_t\ww(\theta),
\end{equation}
where $0<\theta<s$. Therefore, we obtain
\begin{equation}\label{eq:maxdt}
  \max_{0 \le  s \le t_1} \| \partial_t^2 \ww(s)\|_{L^2(\Omega_s)}^2\le \|\partial_t^2\ww(0)\|_{L^2(\Omega_s)}^2+(\dt)^2\max_{0 <\theta<s \le t_1} \| \partial_t^3 \ww(\theta)\|_{L^2(\Omega_s)}^2.
\end{equation}
Applying the same technique in \eqref{eq:maxdt} to $\uu$ and $\lla$, the estimate \eqref{eq:maxdt} becomes
\begin{equation}\label{eq:z1ddtextra}
  \begin{aligned}
    &Z^1(\pdt W,\pdt U,\pdt \Lambda)\\
    &\le C(\dt)^2\|\partial_t^2\ww(0)\|_{L^2(\Omega_s)}^2+C(\dt)^2\|\partial_t^2\uu(0)\|_{L^2(\Omega_f)}^2 \\
    &+C(\dt)^4\max_{0 <\theta<s \le t_1} \| \partial_t^3 \ww(\theta)\|_{L^2(\Omega_s)}^2+C(\dt)^4\max_{0 <\theta<s \le t_1} \| \partial_t^3 \uu(\theta)\|_{L^2(\Omega_f)}^2\\
    &+C\frac{\alpha^2}{\nu_s}\dt( \| \partial_t \uu(0)\|_{L^2(\Sigma)}^2+C (\dt)^4 \max_{0 <\theta<s \le t_1} \| \partial_t^3 \uu(\theta)\|_{L^2(\Sigma)}^2)\\
    &+C\frac{\dt}{\alpha}(\| \partial_t \lla(0)\|_{L^2(\Sigma)}^2+C (\dt)^4 \max_{0 <\theta<s \le t_1} \| \partial_t^3 \lla(\theta)\|_{L^2(\Sigma)}^2)\\
    &+C\frac{\alpha^2}{\nu_s}(\dt)^3\|\pt^2\uu(0)\|^2_{L^2(\Sigma)}+C\frac{(\dt)^3}{\alpha}\|\pt^2\lla(0)\|^2_{L^2(\Sigma)}.
\end{aligned}
\end{equation}

Now we estimate the following term
\begin{equation*}
  \dt F^{2}(\mathfrak{W},\mathfrak{U})=\dt\Big[(-\frac{h_1^2-2h_1^1}{\dt}, \mathfrak{W}^{2})_s+(-\frac{h_2^2-2h_2^1}{\dt}, \mathfrak{U}^{2})_f + \bl \alpha\frac{g_1^2-2g_1^1}{\dt}, \mathfrak{W}^{2} \br + \bl \mathfrak{U}^{2}-\mathfrak{U}^1, \frac{g_2^2-2g_2^1}{\dt} \br\Big]. 
\end{equation*}
Firstly, it follows from the definition of $h_1$ and \eqref{eq:maxdt} that
\begin{equation}
\begin{aligned}
    &\dt(-\frac{h_1^2-2h_1^1}{\dt},\mathfrak{W}^{2})_s\\
    =&\dt(-\frac{h_1^2-h_1^1}{\dt},\mathfrak{W}^{2})_s+(h_1^1,\mathfrak{W}^{2})_s\\
    =&\dt(\pdt^2\ww^2-\pdt(\pt\ww^2),\mathfrak{W}^{2})_s+(h_1^1,\mathfrak{W}^{2})_s\\
    \le& \frac14\|\mathfrak{W}^{2}\|^2_{L^2(\Omega_s)}+C(\dt)^2\|\pdt^2\ww^2-\pdt(\pt\ww^2)\|^2_{L^2(\Omega_s)}+C\|h_1^1\|^2_{L^2(\Omega_s)}\\
    \le& \frac14\|\mathfrak{W}^{2}\|^2_{L^2(\Omega_s)}+C(\dt)^4\max_{0 \le  s \le t_2} \| \partial_t^3 \ww(s)\|_{L^2(\Omega_s)}^2+C(\dt)^2\max_{0 \le  s \le t_1} \| \partial_t^2 \ww(s)\|_{L^2(\Omega_s)}^2\\
    \le& \frac14\|\mathfrak{W}^{2}\|^2_{L^2(\Omega_s)}+C(\dt)^4\max_{0 \le  s \le t_2} \| \partial_t^3 \ww(s)\|_{L^2(\Omega_s)}^2\\
    &+C(\dt)^2\|\partial_t^2\ww(0)\|_{L^2(\Omega_s)}^2+C(\dt)^4\max_{0 <\theta<s \le t_1} \| \partial_t^3 \ww(\theta)\|_{L^2(\Omega_s)}^2.
\end{aligned}
\end{equation}
Similarly, we have
\begin{equation}
\begin{aligned}
    &\dt(-\frac{h_2^2-2h_2^1}{\dt},\mathfrak{U}^{2})_f\\
    \le& \frac14\|\mathfrak{U}^{2}\|^2_{L^2(\Omega_f)}+C(\dt)^4\max_{0 \le  s \le t_2} \| \partial_t^3 \uu(s)\|_{L^2(\Omega_f)}^2\\
    &+C(\dt)^2\|\partial_t^2\uu(0)\|_{L^2(\Omega_f)}^2+C(\dt)^4\max_{0 <\theta<s \le t_1} \| \partial_t^3 \uu(\theta)\|_{L^2(\Omega_f)}^2.
\end{aligned}
\end{equation}
We then have, by the definition of $g_1$,
\begin{equation}\label{eq:g1esti}
\begin{aligned}
    &\alpha\dt\bl\frac{g_1^2-2g_1^1}{\dt},\mathfrak{W}^{2}\br\\
    =&\alpha\dt^2\bl\pdt^2\uu^2, \mathfrak{W}^2\br+\alpha\bl-(\uu^1-\uu^0), \mathfrak{W}^2\br\\
    \le& \frac12\nu_s\dt\|\nabla\mathfrak{W}^{2}\|^2_{L^2(\Sigma)}+C\frac{\alpha^2(\dt)^3}{\nu_s}\|\pdt^2\uu^2\|^2_{L^2(\Sigma)}+C\frac{\alpha^2}{\nu_s\dt}\|\uu^1-\uu^0\|^2_{L^2(\Sigma)}.
\end{aligned}
\end{equation}
Notice that
\begin{equation*}
  \pdt^2\uu^2=\pt^2\uu(0)+\frac{1}{\dt^2}\int_{t_0}^{t_2}(\dt-|r-t_1|)\int_0^r\pt^3\uu(s)\ \!ds\ \!dr,
\end{equation*}
and thus
\begin{equation}
  \|\pdt^2\uu^2\|^2_{L^2(\Sigma)}\le C\|\pt^2\uu(0)\|^2_{L^2(\Sigma)}+C\dt^2\max_{0 \le  s \le t_2} \| \partial_t^3 \uu(s)\|_{L^2(\Sigma)}^2.
\end{equation}
We also have
\begin{equation*}
\uu^1-\uu^0= \dt \partial_t \uu(0) + \int_{0}^{t_1} \int_{0}^r \partial_t^2 \uu(s) ds dr,
\end{equation*}
and therefore 
\begin{equation}
 \| \uu^1-\uu^0\|_{L^2(\Sigma)}^2  \le C \dt^2 \| \partial_t \uu(0)\|_{L^2(\Sigma)}^2+C \dt^4 \max_{0 \le  s \le t_1} \| \partial_t^2 \uu(s)\|_{L^2(\Sigma)}^2.
\end{equation}
Then \eqref{eq:g1esti} becomes
\begin{equation}\label{eq:g1esti1}
\begin{aligned}
    &\alpha\dt\bl\frac{g_1^2-2g_1^1}{\dt},\mathfrak{W}^{2}\br\\
    \le& \frac12\nu_s\dt\|\nabla\mathfrak{W}^{2}\|^2_{L^2(\Sigma)}+C\frac{\alpha^2(\dt)^3}{\nu_s}\big(\|\pt^2\uu(0)\|^2_{L^2(\Sigma)}+\dt^2\max_{0 \le  s \le t_2} \| \partial_t^3 \uu(s)\|_{L^2(\Sigma)}^2\big)\\
    &+C\frac{\alpha^2}{\nu_s}\dt\big( \| \partial_t \uu(0)\|_{L^2(\Sigma)}^2+(\dt)^2 \max_{0 \le  s \le t_1} \| \partial_t^2 \uu(s)\|_{L^2(\Sigma)}^2\big).
\end{aligned}
\end{equation}
By using the Mean Value Theorem again, we improve \eqref{eq:g1esti1} further as
\begin{equation}\label{eq:g1esti2}
\begin{aligned}
    &\alpha\dt\bl\frac{g_1^2-2g_1^1}{\dt},\mathfrak{W}^{2}\br\\
    \le& \frac12\nu_s\dt\|\nabla\mathfrak{W}^{2}\|^2_{L^2(\Sigma)}+C\frac{\alpha^2(\dt)^3}{\nu_s}\big(\|\pt^2\uu(0)\|^2_{L^2(\Sigma)}+\dt^2\max_{0 \le  s \le t_2} \| \partial_t^3 \uu(s)\|_{L^2(\Sigma)}^2\big)\\
    &+C\frac{\alpha^2}{\nu_s}\dt\big(\| \partial_t \uu(0)\|_{L^2(\Sigma)}^2+(\dt)^2\|\partial_t^2\uu(0)\|_{L^2(\Sigma)}^2\big)\\
    &+C\frac{\alpha^2}{\nu_s}\dt^5 \max_{0 <\theta<s \le t_1} \| \partial_t^3 \uu(\theta)\|_{L^2(\Sigma)}^2\\
    \le& \frac12\nu_s\dt\|\nabla\mathfrak{W}^{2}\|^2_{L^2(\Sigma)}+C\frac{\alpha^2(\dt)^3}{\nu_s}\big(\|\pt^2\uu(0)\|^2_{L^2(\Sigma)}+\dt^2\max_{0 \le  s \le t_2} \| \partial_t^3 \uu(s)\|_{L^2(\Sigma)}^2\big)\\
    &+C\frac{\alpha^2}{\nu_s}\dt\big( \| \partial_t \uu(0)\|_{L^2(\Sigma)}^2+(\dt)^4\max_{0 <\theta<s \le t_1} \| \partial_t^3 \uu(\theta)\|_{L^2(\Sigma)}^2\big).
\end{aligned}
\end{equation}

The last term in $\dt F^{2}(\mathfrak{W},\mathfrak{U})$ can be estimated similarly as follows. We have
\begin{equation}\label{eq:u2u1esti}
  \begin{aligned}
    &\dt\bl \mathfrak{U}^{2}-\mathfrak{U}^1, \frac{g_2^2-2g_2^1}{\dt} \br\\
    \le&\frac12\nu_f\dt\|\nabla \mathfrak{U}^2\|^2_{L^2(\Sigma)}+ \frac12\nu_f\dt\|\nabla \mathfrak{U}^1\|^2_{L^2(\Sigma)}\\
    &+C\frac{\alpha^2(\dt)^3}{\nu_f}\big(\|\pt^2\lla(0)\|^2_{L^2(\Sigma)}+\dt^2\max_{0 \le  s \le t_2} \| \partial_t^3 \lla(s)\|_{L^2(\Sigma)}^2\big)\\
    &+C\frac{\alpha^2}{\nu_f}\dt\big( \| \partial_t \lla(0)\|_{L^2(\Sigma)}^2+(\dt)^4\max_{0 <\theta<s \le t_1} \| \partial_t^3 \lla(\theta)\|_{L^2(\Sigma)}^2\big).
  \end{aligned}
\end{equation}
The first term on the right-hand side of the inequality in \eqref{eq:u2u1esti} can be kicked back to $S^2(\mathfrak{W},\mathfrak{U},\mathfrak{L})$, the second term can be estimated similarly by \eqref{eq:z1s1} and \eqref{eq:z1ddtextra}.

The last term $\frac{\dt}{\alpha}\bl\frac{g_2^2-2g_2^1}{\dt},\mathfrak{L}^2\br$ can also be estimated similarly. We have
\begin{equation}\label{eq:frakl2esti}
  \begin{aligned}
    &\frac{\dt}{\alpha}\bl\frac{g_2^2-2g_2^1}{\dt},\mathfrak{L}^2\br\\
    \le&\frac{\dt}{4\alpha}\|\mathfrak{L}^2\|_{L^2(\Sigma)}^2+C\frac{(\dt)^3}{\alpha}\big(\|\pt^2\lla(0)\|^2_{L^2(\Sigma)}+\dt^2\max_{0 \le  s \le t_2} \| \partial_t^3 \lla(s)\|_{L^2(\Sigma)}^2\big)\\
    &+C\frac{1}{\alpha}\dt\big( \| \partial_t \lla(0)\|_{L^2(\Sigma)}^2+(\dt)^4\max_{0 <\theta<s \le t_1} \| \partial_t^3 \lla(\theta)\|_{L^2(\Sigma)}^2\big).
  \end{aligned}
\end{equation}
Combining \eqref{eq:z1ddtextra} to \eqref{eq:frakl2esti}, we have the following estimate for \eqref{eq:zpdt2} under the Assumption \ref{assump:initial1}.
\begin{equation}
  \begin{aligned}
    (\dt)^2Z^2(\pdt^2W,\pdt^2U,\pdt^2\Lambda)+S^2(\mathfrak{W},\mathfrak{U},\mathfrak{L})\le C(\dt)^4\mathsf{Y}_2,
  \end{aligned}
\end{equation}
which gives the desired estimate \eqref{eq:initz2}. Here $\mathsf{Y}_2$ is defined as:
\begin{equation}
  \begin{aligned}
    \mathsf{Y}_2=&\max_{0 <\theta<s \le t_1} \| \partial_t^3 \ww(\theta)\|_{L^2(\Omega_s)}^2+\max_{0 <\theta<s \le t_1} \| \partial_t^3 \uu(\theta)\|_{L^2(\Omega_f)}^2\\
    &+\frac{\alpha^2\dt}{\nu_s}\max_{0 <\theta<s \le t_1} \| \partial_t^3 \uu(\theta)\|_{L^2(\Sigma)}^2+ (\frac{\alpha^2}{\nu_f}+\frac{1}{\alpha})\dt\max_{0 <\theta<s \le t_1} \| \partial_t^3 \lla(\theta)\|_{L^2(\Sigma)}^2\\
    &+\max_{0 \le  s \le t_2} \| \partial_t^3 \ww(s)\|_{L^2(\Omega_s)}^2+\max_{0 \le  s \le t_2} \| \partial_t^3 \uu(s)\|_{L^2(\Omega_f)}^2\\
    &+\frac{\alpha^2\dt}{\nu_s}\max_{0 \le  s \le t_2} \| \partial_t^3 \uu(s)\|_{L^2(\Sigma)}^2+(\frac{\alpha^2}{\nu_f}+\frac{1}{\alpha})\dt\max_{0 \le  s \le t_2} \| \partial_t^3 \lla(s)\|_{L^2(\Sigma)}^2\\
    &+\frac{\alpha^2}{\nu_s}+\frac{\alpha^2}{\nu_f}+\frac{1}{\alpha}+1.
  \end{aligned}
\end{equation}
\end{proof}

\section{Sketch of Proof: Corollary \ref{Corrate}}\label{appdix:fcoro}
\begin{proof}[Proof of \eqref{eq:errorW} in Corollary \ref{Corrate}]
  To prove \eqref{eq:errorW}, it is suffice to bound the term $\Xi_0^N(h_1, h_2, \alpha g_1-g_2,\tilde{g}_2)$. Note that since $g_2^{n+1}=\tilde{g}_2^{n+1}$ on $\Sigma$, we have
  \begin{alignat*}{1}
 \Xi_0^N(h_1, h_2, \alpha g_1-g_2,\tilde{g}_2):=  &\dt \sum_{n=0}^{N-1} (\frac{1}{\nu_s}\| h_1^{n+1}\|_{L^2(\Omega_s)}^2+(\frac{1}{\nu_f}+\frac{1}{\alpha})\| h_2^{n+1}\|_{L^2(\Omega_f)}^2) \\
&+  \dt \sum_{n=1}^{N-1} \frac{1}{\nu_f\alpha^2} \| \pdt \tilde{g}_2^{n+1}\|_{L^2(\Omega_f)}^2 +\dt \sum_{n=0}^{N-1} \Big(  \frac{\nu_f}{\alpha^2} \| \nabla \tilde{g}_2^{n+1}\|_{L^2(\Omega_f)}^2 + \frac{1}{\alpha}\| \tilde{g}_2^{n+1} \|_{L^2(\Omega_f)}^2 \Big)  \\
&+ \dt \sum_{n=0}^{N-1} \Big(\frac{1}{\nu_s} \| \alpha g_1\|_{L^2(\Sigma)}^2+\frac{1}{\nu_f}\| \tilde{g}_2^{n+1}\|_{L^2(\Sigma)}^2 \Big) + \frac{1}{\alpha^2}\|\tilde{g}_2^N\|_{L^2(\Omega_f)}^2 \\
&=T_1+T_2+\ldots+T_8.
\end{alignat*}
All the terms in $\Xi_0^N(h_1, h_2, \alpha g_1-g_2,\tilde{g}_2)$ can be easily estimated by \eqref{eq:tdesti} except the $T_3$ and $T_4$. For $T_3$, it follows from \eqref{621} that,
\begin{equation}
\begin{aligned}
    \dt \sum_{n=1}^{N-1} \frac{1}{\nu_f\alpha^2} \| \pdt \tilde{g}_2^{n+1}\|_{L^2(\Omega_f)}^2&=(\dt)^3 \sum_{n=1}^{N-1} \frac{1}{\nu_f\alpha^2} \| \pdt^2 \tilde{\lla}^{n+1}\|_{L^2(\Omega_f)}^2\\
    &\le C\frac{(\dt)^2}{\nu_f\alpha^2}\|\pt^2 \tilde{\lla}\|_{L^2(0,T;L^2(\Omega_f))}^2\\
    &\le C(\dt)^2\frac{\nu_f}{\alpha^2}\|\pt^2 \uu\|_{L^2(0,T;H^1(\Omega_f))}^2.
\end{aligned}
\end{equation}
The term $T_4$ can be bounded as follows,
$$
 \dt\frac{\nu_f}{\alpha^2} \sum_{n=0}^{N-1}\|\nabla \tilde{g}_2^{n+1}\|_{L^2(\Omega_f)}^2 
\le C \dt\frac{(\nu_f)^3}{\alpha^2} \sum_{n=0}^{N-1}  \Big( \|\nabla(\uu^{n+1}-u^n)\|_{L^2(\Omega_f)}^2 + \|D^2 (\uu^{n+1}-\uu^n)\|_{L^2(\Omega_f)}^2 \Big).
$$
Here $D^2 \uu$ denotes the Hessian of $\uu$. Using \eqref{721} we get 
\begin{alignat*}{1}
\dt\frac{\nu_f}{\alpha^2} \sum_{n=0}^{N-1}\|\nabla \tilde{g}_2^{n+1}\|_{L^2(\Omega_f)}^2 
 \le & C  (\dt)^2\frac{(\nu_f)^3}{\alpha^2}  \|\pt \uu\|_{L^2(0,T;H^2(\Omega_f))}^2.
\end{alignat*} 
The estimates above imply \eqref{eq:errorW}.
\end{proof}
Before we prove the estimate \eqref{eq:errordW}, we need the following preliminary results. We use the Bochner norms $\|v\|_{L^2(a,b; X)}=\Big(\int_a^b \|v(\cdot, s)\|_X^2 ds\Big)^{1/2}$ and $\|v\|_{L^\infty(a,b; X)}=\text{ess sup}_{ a \le  s \le b}  \|v(\cdot, s)\|_X$. For a Sobolev space $X$, it is well known that
\begin{subequations}
\begin{alignat}{1}\label{721}
\| v^{n+1}-v^n\|_X^2 \le&  C \dt \int_{t_n}^{t_{n+1}} \|\pt v(\cdot,s)\|_{X}^2 ds, \\
\| \pdt v^{n+1}-\pt v^{n+1}\|_X^2 \le&  C \dt \int_{t_n}^{t_{n+1}} \|\pt^2 v(\cdot,s)\|_{X}^2 ds, \label{eq:tdesti}\\
\int_{a}^{b} \|v(\cdot,s)\|_X^2 ds \leq & (b-a) \|v\|_{L^\infty(a,b;X)}^2.\label{eq:linftyesti}
\end{alignat}
\end{subequations}
The following identities can easily be shown
\begin{equation}\label{eq:secordfd}
 \pdt^2 v^n= \frac{1}{(\dt)^2} \int_{-\dt}^{\dt} (\dt -|s|)\pt^2 v(\cdot, t_n+s) ds,
\end{equation}
and
\begin{equation}\label{eq:cg23}
  \begin{aligned}
 \frac{\pdt^2 v^{n}- \pdt^2 v^{n-1}}{\dt} =& \frac{1}{(\dt)^3} \int_{-\dt}^{\dt} (\dt -|s|)\big(\pt^2 v(\cdot, t_n+s)- \pt^2 v(\cdot, t_{n-1}+s)\big) ds\\
 =& \frac{1}{(\dt)^3} \int_{-\dt}^{\dt} (\dt -|s|)\int_{t_{n-1}+s}^{t_{n}+s} \pt^3 v(\cdot, r) dr ds. 
\end{aligned}
\end{equation}
From these we can show that
\begin{equation}\label{621}
\| \pdt^2 v^n\|_{X}^2 \le  \frac{C}{\dt} \int_{t_{n-1}}^{t_{n+1}} \|\pt^2 v(\cdot,s)\|_{X}^2 ds, 
\end{equation}
and
\begin{alignat}{1}
 \|\frac{\pdt^2 v^{n}- \pdt^2 v^{n-1}}{\dt}\|_X^2 \le \frac{C}{\dt} \int_{t_{n-2}}^{t_{n+1}} \|\pt^3 v(\cdot, r)\|_X^2 dr.    \label{v3d}
\end{alignat}
\begin{proof}[Proof of \eqref{eq:errordW} in Corollary \ref{Corrate}]
It is similar to prove \eqref{eq:errordW}. Indeed, let us define, for $j=1,2$
\begin{equation*}
  G_j^{n+1}=\pdt g_j^{n+1},\quad\tilde{G}_2^{n+1}= \pdt \tilde{g}_2^{n+1}, \quad H_j^{n+1}=\pdt h_j^{n+1}.
\end{equation*}
We then need to bound the term
\begin{alignat*}{1}
\Xi_1^N&(H_1, H_2, \alpha G_1-G_2, \tilde{G}_2)\\
=&  \dt \sum_{n=1}^{N-1} (\frac{1}{\nu_s}\| H_1^{n+1}\|_{L^2(\Omega_s)}^2+(\frac{1}{\nu_f}+\frac{1}{\alpha})\| H_2^{n+1}\|_{L^2(\Omega_f)}^2) \\
&+  \dt \sum_{n=2}^{N-1}\frac{1}{\nu_f\alpha^2} \| \pdt \tilde{G}_2^{n+1}\|_{L^2(\Omega_f)}^2 +\dt \sum_{n=1}^{N-1} \Big( \frac{\nu_f}{\alpha^2} \| \nabla \tilde{G}_2^{n+1}\|_{L^2(\Omega_f)}^2 + \frac{1}{\alpha}\| \tilde{G}_2^{n+1} \|_{L^2(\Omega_f)}^2 \Big)  \\
&+ \dt \sum_{n=1}^{N-1} \Big(\frac{1}{\nu_s} \| \alpha G_1\|_{L^2(\Sigma)}^2+\frac{1}{\nu_f}\| \tilde{G}_2^{n+1}\|_{L^2(\Sigma)}^2 \Big) +\frac{1}{\alpha^2}\|\tilde{G}_2^N\|_{L^2(\Omega_f)}^2 \\
=&R_1+R_2+\ldots+R_8.
\end{alignat*}
For $R_1$, it follows from the definitions of $H_1^{n+1}$ and $h_1^{n+1}$ that,
  \begin{equation}\label{eq:h1esti}
    \begin{aligned}
      \|H_1^{n+1}\|^2_{L^2(\Omega_s)}&=\int_{\Omega_s}\left(\frac{h_1^{n+1}-h_1^n}{\dt}\right)^2\ \!dx\\
      &=\int_{\Omega_s}\left(\frac{\pt \ww^{n+1} - \pdt \ww^{n+1}-\pt \ww^{n} + \pdt \ww^{n}}{\dt}\right)^2\ \!dx\\
      &=\int_{\Omega_s}\left(\pdt (\partial_t\ww)^{n+1}-\pdt^2\ww^{n}\right)^2\ \!dx\\
      &=\int_{\Omega_s}\left(\pdt (\partial_t\ww)^{n+1}-\pt^2\ww^n+\pt^2\ww^n-\pdt^2\ww^{n+1}\right)^2\ \!dx\\
      &\le C\int_{\Omega_s}\left(\frac{1}{\dt}\int_{t_n}^{t_{n+1}}(\pt^2\ww(t_n)-\pt^2\ww(s))\ \!ds\right)^2\ \!dx\\
      &\quad +C \int_{\Omega_s}\left(\frac{1}{(\dt)^2}\int_{t_{n-1}}^{t_{n+1}}\big(\dt-|s-t_n|\big)\big(\pt^2\ww(s)-\pt^2\ww(t_n)\big)\ \!ds\right)^2\ \!dx\\
      &\le C\dt\|\pt^3\ww\|^2_{L^2((t_{n-1},t_{n+1}),L^2(\Omega_s))}.
    \end{aligned}
  \end{equation}
  Therefore, we obtain
  \begin{equation}\label{eq:dth1}
    \dt\sum_{n=1}^{N-1}\frac{1}{\nu_s}\|H_1^{n+1}\|^2_{L^2(\Omega_s)}\le C(\dt)^2\frac{1}{\nu_s}\|\pt^3\ww\|^2_{L^2((0,T),L^2(\Omega_s))}.
  \end{equation}
  The estimate of $R_2$ is similar to that of $R_1$. The term $R_3$ can be estimated as follow by \eqref{v3d},
  \begin{equation}
    \begin{aligned}
      \dt \sum_{n=2}^{N-1} \frac{1}{\nu_f\alpha^2} \| \pdt \tilde{G}_2^{n+1}\|_{L^2(\Omega_f)}^2&=(\dt)^3\sum_{n=2}^{N-1} \frac{1}{\nu_f\alpha^2}\| \frac{\pdt^2 \tilde{\lla}^{n}- \pdt^2 \tilde{\lla}^{n-1}}{\dt}\|_{L^2(\Omega_f)}^2\\
      &\le C\frac{(\dt)^2}{\nu_f\alpha^2}\|\pt^3 \tilde{\lla}\|_{L^2(0,T;L^2(\Omega_f))}^2\\
    &\le C(\dt)^2\frac{\nu_f}{\alpha^2}\|\pt^3 \uu\|_{L^2(0,T;H^1(\Omega_f))}^2.
    \end{aligned}
  \end{equation}
For $R_4$, it follows from the definition of $\tilde{\lla}$ and the fact $\|\nabla\phi\|_{L^2(\Omega_f)}\le C$ that,
  \begin{equation}\label{eq:ng2}
    \begin{aligned}
      \|\nabla\tilde{G}_2^{n+1}\|_{L^2(\Omega_f)}^2&=\int_{\Omega_f}\left(\frac{\nabla\tilde{\lla}^{n+1}-2\nabla\tilde{\lla}^{n}+\nabla\tilde{\lla}^{n-1}}{\dt}\right)^2\ \!dx\\
      &=\int_{\Omega_f}\left(\frac{1}{\dt}\int_{t_{n-1}}^{t_{n+1}}\big(\dt-|s-t_n|\big)\pt^2\nabla\tilde{\lla}(s)\ \!ds\right)^2\ \!dx\\
      &\le C\int_{\Omega_f}\left(\int_{t_{n-1}}^{t_{n+1}}|\pt^2\nabla\tilde{\lla}(s)|\ \!ds\right)^2\ \!dx\\
      &\le C\dt\int_{\Omega_f}\int_{t_{n-1}}^{t_{n+1}}(\pt^2\nabla\tilde{\lla}(s))^2\ \!ds\ \!dx\\
      &\le C \nu_f^2\dt \|\pt^2\uu\|_{L^2(t_{n-1},t_{n+1};H^2(\Omega_f))}^2,
    \end{aligned}
  \end{equation}
  and hence,
  \begin{equation}\label{eq:R4}
    \dt \sum_{n=1}^{N-1}\frac{\nu_f}{\alpha^2}\| \nabla \tilde{G}_2^{n+1}\|_{L^2(\Omega_f)}^2\le C(\dt)^2\frac{(\nu_f^3)}{\alpha^2}\|\pt^2\uu\|_{L^2(0,T;H^2(\Omega_f))}^2.
  \end{equation}
  The remaining terms $R_5$ to $R_8$ can be estimated similarly, we give the estimate of $R_6$ here:
  \begin{equation}\label{eq:g1s}
    \begin{aligned}
      \|G_1^{n+1}\|^2_{L^2(\Sigma)}&=\int_{\Sigma}\left(\frac{g_1^{n+1}-g_1^n}{\dt}\right)^2\ \!d\sigma\\
      &=\int_{\Sigma}\left(\frac{\uu^{n+1}-2\uu^{n}+\uu^{n-1}}{\dt}\right)^2\ \!d\sigma\\
      &=\int_{\Sigma}\left(\frac{1}{\dt}\int_{t_{n-1}}^{t_{n+1}}\big(\dt-|s-t_n|\big)\pt^2\uu(s)\ \!ds\right)^2\ \!d\sigma\\
      &\le C\dt\|\pt^2\uu\|^2_{L^2((t_{n-1},t_{n+1}),L^2(\Sigma))},
    \end{aligned}
  \end{equation}
  and therefore, we obtain
  \begin{equation}\label{eq:R6}
    \dt \sum_{n=1}^{N-1} \frac{1}{\nu_s} \| \alpha G_1\|_{L^2(\Sigma)}^2\le C(\dt)^2\frac{\alpha^2}{\nu_s}\|\pt^2\uu\|^2_{L^2((0,T),L^2(\Sigma))}.
  \end{equation}
 The last term $\|\pdt \tilde{g}_2^{2}\|_{L^2(\Omega_f)}^2$ can be estimated similarly as \eqref{eq:ng2} by using \eqref{eq:linftyesti} as follows. We have
  \begin{equation}\label{eq:pdtg2}
  \begin{aligned}
        \|\pdt \tilde{g}_2^{2}\|_{L^2(\Omega_f)}^2&\le C\nu_f^2\dt\|\pt^2 \uu\|_{L^2(t_0,t_2;H^1(\Omega_f))}^2\\
        &\le C\nu_f^2(\dt)^2\|\pt^2 \uu\|_{L^\infty(0,T;H^1(\Omega_f))}^2.
  \end{aligned}
  \end{equation}
\end{proof}

In order to prove \eqref{eq:errorddW}, we define the following quantities for $j=1,2$,
\begin{equation}\label{eq:curlygh}
  \cG_j^{n+1}=\pdt G_j^{n+1},\quad \cH_j^{n+1}=\pdt H_j^{n+1}, \quad \tilde{\cG}_2^{n+1}= \pdt \tilde{G}_2^{n+1}.
\end{equation}
Note that we also have the following
\begin{equation}
  \begin{aligned}\label{eq:cg21dd}
\tilde{\cG}_2^{n+1}=\frac{\tilde{g}_2^{n+1}-2\tilde{g}_2^{n}+\tilde{g}_2^{n-1}}{(\dt)^2}=\frac{\tilde{\lla}^{n+1}-3\tilde{\lla}^{n}+3\tilde{\lla}^{n-1}-\tilde{\lla}^{n-2}}{(\dt)^2}.
\end{aligned}
\end{equation}
Denote
\begin{equation}\label{eq:dfq3}
  \pdt^3v^{n+1}=\frac{v^{n+1}-3v^n+3v^{n-1}-v^{n-2}}{(\dt)^3},
\end{equation}
thus we see that 
\begin{equation}
  \begin{aligned}\label{eq:cg22dd}
\frac{\tilde{\cG}_2^{n+1}-\tilde{\cG}_2^{n}}{(\dt)^2}=\frac{\pdt^3\tilde{\lla}^{n+1}-\pdt^3\tilde{\lla}^{n}}{\dt}.
\end{aligned}
\end{equation}
Moreover, we see that 
\begin{equation*}
\begin{aligned}\label{eq:cg23dd}
\frac{\pdt^3 v^{n}- \pdt^3 v^{n-1}}{\dt} =& \frac{1}{(\dt)^4} \Bigg(\int_{-\dt}^{\dt} (\dt -|s|)\big(\pt^2 v(\cdot, t_n+s)- \pt^2 v(\cdot, t_{n-1}+s)\big) ds\\
 &-\int_{-\dt}^{\dt} (\dt -|s|)\big(\pt^2 v(\cdot, t_{n-1}+s)- \pt^2 v(\cdot, t_{n-2}+s)\big) ds\Bigg)\\
 =&\frac{1}{(\dt)^4} \Bigg(\int_{-\dt}^{\dt} (\dt -|s|)\big(\pt^2 v(\cdot, t_n+s)- 2\pt^2 v(\cdot, t_{n-1}+s)+\pt^2 v(\cdot, t_{n-2}+s)\big) ds\Bigg)\\
 =& \frac{1}{(\dt)^4} \int_{-\dt}^{\dt} (\dt -|s|)\int_{t_{n-2}+s}^{t_{n}+s} (\dt-|r-t_{n-1}|)\pt^4 v(\cdot, r) dr ds.
\end{aligned}  
\end{equation*}

Therefore, we obtain the following estimate
\begin{alignat}{1}
 \|\frac{\pdt^3 v^{n}- \pdt^3 v^{n-1}}{\dt}\|_X^2 \le \frac{C}{\dt} \int_{t_{n-3}}^{t_{n+1}} \|\pt^4 v(\cdot, r)\|_X^2 dr.    \label{v3dd}
\end{alignat}
\begin{proof}[Proof of \eqref{eq:errorddW} in Corollary \ref{Corrate}]
  We now bound the term $\Xi^N(\pdt^2 h_1, \pdt^2 h_2, \pdt^2 (\alpha g_1-g_2), \pdt^2 \tilde{g}_2)$ which is $\Xi^N(\cH_1, \cH_2, (\alpha \cG_1-\cG_2),  \tilde{\cG}_2)$ according to \eqref{eq:curlygh}. The techniques are similar to those mentioned above. We present the key estimates involving $\cH_1$ and $\tilde{\cG}_2$ first.
  It follows from the definition of $\cH_1^{n+1}$ that,
  \begin{equation}\label{eq:h1estid}
    \begin{aligned}
      \|\cH_1^{n+1}\|^2_{L^2(\Omega_s)}&=\int_{\Omega_s}\left(\frac{h_1^{n+1}-2h_1^n+h_1^{n-1}}{(\dt)^2}\right)^2\ \!dx\\
      &=\int_{\Omega_s}\left(\frac{\pdt (\pt\ww)^{n+1}-\pdt^2\ww^{n+1}-(\pdt (\pt\ww)^{n}-\pdt^2\ww^{n})}{\dt}\right)^2\ \!dx\\
      &=\int_{\Omega_s}\left(\pdt^2(\pt\ww)^{n+1}-\pdt^3\ww^{n+1}\right)^2\ \!dx\\
      &=\int_{\Omega_s}\left(\pdt^2(\pt\ww)^{n+1}-\pt^3\ww^n+\pt^3\ww^n-\pdt^3\ww^{n+1}\right)^2\ \!dx\\
      &\le C\int_{\Omega_s}\left(\frac{1}{(\dt)^2}\int_{t_{n-1}}^{t_{n+1}}\big(\dt-|s-t_n|\big)\big(\pt^3\ww(s)-\pt^3\ww(t_n)\big)\ \!ds\right)^2\ \!dx\\
      &\quad+ C\int_{\Omega_s}\left(\frac{1}{(\dt)^3}\int_{-\dt}^{\dt}(\dt-|s|)\int_{t_{n-1}+s}^{t_n+s}\pt^3\ww(t_n)-\pt^3\ww(\cdot,r)\ \!ds\right)^2\ \!dx\\
      &\le C\dt\|\pt^4\ww\|^2_{L^2((t_{n-2},t_{n+1}),L^2(\Omega_s))}.
    \end{aligned}
  \end{equation}
  We also have, according to the definition of $\tilde{\lla}$ and the fact $\|\nabla\phi\|_{L^2(\Omega_f)}\le C$ that,
  \begin{equation}\label{eq:frakg2}
    \begin{aligned}
      \|\nabla\tilde{\cG}_2^{n+1}\|_{L^2(\Omega_f)}^2&=\int_{\Omega_f}\left(\frac{\frac{\nabla\tilde{\lla}^{n+1}-2\nabla\tilde{\lla}^{n}+\nabla\tilde{\lla}^{n-1}}{\dt}-\frac{\nabla\tilde{\lla}^{n}-2\nabla\tilde{\lla}^{n-1}+\nabla\tilde{\lla}^{n-2}}{\dt}}{\dt}\right)^2\ \!dx\\
      &=\int_{\Omega_f}\left(\frac{1}{(\dt)^2}\int_{-\dt}^{\dt} (\dt -|s|)\big(\pt^2 \nabla\tilde{\lla}(\cdot, t_n+s)- \pt^2 \nabla\tilde{\lla}(\cdot, t_{n-1}+s)\big) ds\right)^2\ \!dx\\
      &\le C\int_{\Omega_f}\left(\frac{1}{\dt}\int_{-\dt}^{\dt}|\int_{t_{n-1}+s}^{t_{n}+s} \pt^3 \nabla\tilde{\lla}(\cdot, r) dr |\ \!ds\right)^2\ \!dx\\
      &\le C\dt\int_{\Omega_f}\int_{t_{n-2}}^{t_{n+1}}(\pt^3\nabla\tilde{\lla}(\cdot,r))^2\ \!dr\ \!dx\\
      &\le C \nu_f^2\dt \|\pt^3\uu\|_{L^2(t_{n-2},t_{n+1};H^2(\Omega_f))}^2.
    \end{aligned}
  \end{equation}
  It follows from \eqref{v3dd} that,
  \begin{equation}
    \begin{aligned}
      \dt \sum_{n=3}^{N-1} \frac{1}{\nu_f\alpha^2}\| \pdt \tilde{\cG_2}^{n+1}\|_{L^2(\Omega_f)}^2&=(\dt)^3 \sum_{n=3}^{N-1} \frac{1}{\nu_f\alpha^2} \|\frac{\pdt^3 \tilde{\lla}^{n}- \pdt^3 \tilde{\lla}^{n-1}}{\dt}\|_{L^2(\Omega_f)}^2\\
      &\le C\frac{(\dt)^2}{\nu_f\alpha^2}\|\pt^4 \tilde{\lla}\|_{L^2(0,T;L^2(\Omega_f))}^2\\
    &\le C(\dt)^2\frac{\nu_f}{\alpha^2}\|\pt^4 \uu\|_{L^2(0,T;H^1(\Omega_f))}^2.
    \end{aligned}
  \end{equation}

  The last term $\|\pdt^2 \tilde{g}_2^{3}\|_{L^2(\Omega_f)}^2$ can be estimated similarly as \eqref{eq:frakg2} by using \eqref{eq:linftyesti} as follows. We have
  \begin{equation}\label{eq:pdt2g2}
  \begin{aligned}
        \|\pdt^2 \tilde{g}_2^{3}\|_{L^2(\Omega_f)}^2&\le C\nu_f^2\dt\|\pt^3 \uu\|_{L^2(t_0,t_3;H^1(\Omega_f))}^2\\
        &\le C\nu_f^2(\dt)^2\|\pt^3 \uu\|_{L^\infty(0,T;H^1(\Omega_f))}^2.
  \end{aligned}
  \end{equation}

  Then we follow the same idea as before and obtain the following bound:
  \begin{equation}
    \Xi_2^N(\cH_1, \cH_2, (\alpha \cG_1-\cG_2),  \tilde{\cG}_2)\le C(\dt)^2\mathfrak{Y},
  \end{equation}
  where $\mathfrak{Y}$ is defined in \eqref{eq:frakY}.
  \end{proof}

\section{Sketch of Proof: Corollary \ref{H2errordtexplcit}}\label{appdix:fcoro1}

\begin{proof}
According to Corollary \ref{H2errordt}, we need to bound the terms on the right-hand side of the inequality \eqref{eq:h2duesti}.
 The analysis to bound the first term  $\Xi_1^N(\pdt \p_x  h_1, \pdt \p_x h_2, \pdt \p_x (\alpha g_1-g_2),  \pdt \p_x \tilde{g}_2)$ is almost identical to that of the term $\Xi_1^N(H_1, H_2, \alpha G_1-G_2, \tilde{G}_2)$ except the additional partial derivative which does not affect the techniques. Therefore, we obtain the bound
  \begin{equation}
    \Xi_1^N(\pdt \p_x  h_1, \pdt \p_x h_2, \pdt \p_x (\alpha g_1-g_2),  \pdt \p_x \tilde{g}_2)\le C(\dt)^2\mathbb{Y},
  \end{equation}
  where $\mathbb{Y}$ is defined in \eqref{eq:bbY}.
  The second term in \eqref{eq:h2duesti} is bounded in Corollary \ref{Corrate}. The third term, the sixth term and the last term can be estimated similarly by \eqref{eq:pdtg2}, \eqref{eq:pdt2g2} and \eqref{eq:dth1}. The term $Z^{2}(\pdt^2 W, \pdt^2 U, \pdt^2 \Lla)$ is estimated in \eqref{eq:initz2}. The term $Z^{1}(\pdt \p_xW, \pdt \p_xU, \pdt \p_x\Lla)$ can be estimated similarly to \eqref{eq:initz1} under Assumption \ref{assump:initial2}. This finishes the proof.
\end{proof}

\section*{Declarations} 

\noindent{\bf Ethical approval}: Not Applicable.
\\

\noindent{\bf Availability of supporting data}: No additional data are available.
\\

\noindent{\bf Competing interests}: The authors declare no competing interests.
\\

\noindent{\bf Funding}:  This material is based upon work supported by the National Science Foundation under Grant No. DMS-1929284 while the last author was in residence at the Institute for Computational and Experimental Research in Mathematics in Providence, RI, during the "Numerical PDEs: Analysis, Algorithms, and Data Challenges" program. EB was supported by EPSRC grants EP/J002313/2 and EP/W007460/1. EB and MF were supported by the IMFIBIO Inria-UCL associated team.  
\\

\noindent{\bf Authors' contributions}: All the authors contributed significantly.
\\

\noindent{\bf Acknowledgement}: Not Applicable.

\bibliographystyle{abbrv}
\bibliography{referencesWP}

\end{document}